\documentclass[11 pt]{amsart}
\usepackage{amssymb}
\usepackage{amsmath}
\usepackage{amsfonts}
\usepackage{graphicx}
\usepackage{amsthm}
\usepackage{enumerate}
\usepackage[mathscr]{eucal}
\usepackage{verbatim}
\usepackage{color}

\theoremstyle{plain}
\newtheorem{theorem}{Theorem}[section]
\newtheorem{lemma}[theorem]{Lemma}
\newtheorem{lem}[theorem]{Lemma}
\newtheorem{proposition}[theorem]{Proposition}

\newtheorem{defn}[theorem]{Definition}

\numberwithin{equation}{section}
\theoremstyle{plain}

\numberwithin{equation}{section}
\theoremstyle{remark}

\newcommand{\Real}{\mathbb R}

\newcommand{\zbh}[1]{z_{j{#1}}}

\newtheorem{remark}[theorem]{Remark}

\def\bbR{{\mathbb {R}}}
\def\bbZ{{\mathbb {Z}}}
\def\la{\lambda}
\def\la{\lambda}

\def\lc{\lesssim}

\def\intslash{\rlap{\kern  .32em $\mspace {.5mu}\backslash$ }\int}
\def\qsl{{\rlap{\kern  .32em $\mspace {.5mu}\backslash$ }\int_{Q_x}}}
\def\Re{\operatorname{Re\,}}
\def\Im{\operatorname{Im\,}}

\def\emph#1{{\it #1 }}

\def\lc{\lesssim}
\def\gc{\gtrsim}

\def\eps{\varepsilon}

\def\la{\lambda}

\def\bbC{{\mathbb {C}}}

\def\bbN{{\mathbb {N}}}

\def\bbR{{\mathbb {R}}}

\def\bbZ{{\mathbb {Z}}}

\font \roman = cmr10 at 10 true pt

\def\z{{\hbox{\roman z}}}

\def\be#1{\begin{equation}\label{#1}}
\def\ee{\end{equation}}
\def\bas{\begin{align*}}
\def\eas{\end{align*}}
\def\bi{\begin{itemize}}
\def\ei{\end{itemize}}

\def\eps{\varepsilon}

\def\emph#1{{\it #1}}
\def\textbf#1{{\bf #1}}
\def\intslash{\rlap{\kern  .32em $\mspace {.5mu}\backslash$ }\int}
\def\qsl{{\rlap{\kern  .32em $\mspace {.5mu}\backslash$ }\int_{Q_x}}}

\begin{document}
\date{\today}

\title
[Restriction estimate for complex curves] {Restriction of the Fourier
transform to some complex curves}

\author[]
{Jong-Guk Bak  and  Seheon Ham}

\address {Department of Mathematics\\ Pohang University of Science and Technology
\\ Pohang 790-784, Korea} \email{bak@postech.ac.kr, beatles8@postech.ac.kr}


\thanks{
J.-G. Bak was supported in part by National Research Foundation grant 2010-0024861
from the Ministry of Education, Science and Technology of Korea}

\begin{abstract}
The purpose of this paper is to prove a Fourier restriction estimate for certain 2-dimensional
surfaces in $\bbR^{2d}$, $d\ge 3$. These surfaces are defined by a complex curve $\gamma(z)$ of simple type, which is given by a mapping of the form
\[ z\mapsto \gamma (z) = \Big(z, \, z^2,\cdots, \, z^{d-1}, \, \phi(z) \Big) \]
where $\phi(z)$ is an analytic function on a domain $\Omega \subset \bbC$. This is regarded as a real mapping $z=(x,y) \mapsto \gamma(x,y)$ from $\Omega \subset \bbR^2$ to $\bbR^{2d}$.

Our results cover the case $\phi(z) = z^N$ for any nonnegative integer $N$, in all dimensions $d\ge 3$. The main result is a uniform estimate, valid when $d=3$, where $\phi(z)$ may be taken to be an arbitrary polynomial of degree at most $N$. It is uniform in the sense that the operator norm is independent of the coefficients of the polynomial. These results are analogues of the uniform restricted strong type estimates in \cite{BOS3}, valid for polynomial curves of simple type and some other classes of curves in $\bbR^d$, $d\ge 3$.
\end{abstract}

\maketitle

\section{Introduction and statement of results}
\label{intro}
Let $t\mapsto \gamma(t)$ be a curve in $\bbR^d$, defined on an interval $I$.
Let us consider a Fourier restriction estimate of the following form:
\begin{equation} \label{affrestr-0}
\Big(\int_I
 |\widehat f (\gamma (t))|^{q} \, w(t)\, dt\Big)^{1/q} \le
C \|f\|_{L^{p}(\bbR^d)}
\end{equation}
where
$\widehat f (\xi)$ denotes the Fourier transform of $f \in L^p(\bbR^d)$ and
\begin{equation} \label{theweight}
 w (t) = |\tau(t)|^{\frac{2}{d^2+d}}, \text{ with } \tau(t) = \det(\gamma'(t), \cdots,
\gamma^{(d)}(t)) .
\end{equation}
%
%
Here, the measure $w(t)\, dt$ is called the `affine arclength measure' ({\it cf.} \cite{DM1,DM2,BOS1}). We are mostly interested in proving {\it uniform} estimates for \eqref{affrestr-0}, that is, we would like to take the constant $C$ to be uniform over given classes of curves. Also, whenever appropriate we would like to prove {\it global} estimates, that is, for $I = \bbR$ or $(0,\infty)$.

For the interesting history of this problem we refer the reader to \cite{D2, BOS1, BOS3} and the references therein.
The endpoint versions of the Fourier restriction estimates \eqref{affrestr-0} for some
classes of curves were established in \cite{BOS3}. We shall now describe two
such results. The first concerns the case of `monomial' curves
of the form
\begin{equation}\label{monomial}
t\mapsto \gamma_a(t)=(t^{a_1},\, t^{a_2},\cdots,\, t^{a_d}), \quad
0<t<\infty
\end{equation}
where $a=(a_1,\dots, a_d)$ is a $d$-tuple of arbitrary real numbers. For $d\ge 2$, let $p_d = (d^2+d+2)/(d^2+d)$.
The endpoint result
may be stated as follows:

\begin{theorem}{\rm (\cite{BOS3})} \label{powerthm} Let $w(t)\, dt = w_a(t)\, dt$ denote the affine arclength measure for the curve \eqref{monomial}, where $w (t)$ is given by \eqref{theweight} with $\gamma = \gamma_a$.
Then, for $d\ge 3$, there is a constant $C(d)<\infty$ such that for all
$f\in L^{p_d ,1}(\bbR^d)$,
\begin{equation} \label{affrestr-1}
\Big(\int_0^\infty
 |\widehat f (\gamma_a(t))|^{p_d} \, w_a(t) \, dt\Big)^{1/p_d} \le
C(d)\|f\|_{L^{p_d ,1}(\bbR^d)}.
\end{equation}
\end{theorem}
The constant in \eqref{affrestr-1} is uniform in the sense that
it does not depend on $a_1, a_2, \cdots, a_d$. We would like to point out that the versions of \eqref{affrestr-1} fail when $d=2$ (for $p_2 = 4/3$), even in the nondegenerate case and even when the target space is replaced by $L^1(I; wdt)$ for a finite interval $I$. (See \cite{bcss}; see also Section 1 in \cite{BOS1}.)

The $(L^p, L^q)$ estimates, in the optimal range $1\le p < p_d$, $q=2p'/(d^2+d)$, follow by interpolating \eqref{affrestr-1} and the $(L^1, L^\infty)$ estimate. These estimates were proved earlier in \cite{BOS1}, following the work in \cite{DM2}. (For a general result in the 2-dimensional case see, for instance, \cite{Sj} and the references therein.)

Similar results have been proved for some other classes of curves including the polynomial curves of `simple'
type given by
\begin{equation} \label{polcurve}
\Gamma_b(t)= \Big(t, \, t^2 ,\cdots,\, t^{d-1}
, \,P_b(t)\Big), \quad t\in \bbR
\end{equation}
%
%
in $\bbR^d$, where $P_b$  is an arbitrary polynomial of
degree $N\ge 0$, with the coefficients $(b_0, \cdots, b_N)=b\in
\bbR^{N+1}$. Namely,
$ P_b(t) = \sum_{j=0}^N b_j t^j$.
The affine arclength measure is given by $W_b(t) \, dt$,
where
$W_b(t)= |\tau(t)|^{2/(d^2+d)} = |c_d \, P_b^{(d)}(t)|^{2/(d^2+d)}$ with $c_d = 2! \cdots (d-1)!$.
The endpoint estimate in this case is the following
\begin{theorem}{\rm (\cite{BOS3})} \label{polthm}
For $d\ge 3$, there is a constant $C(N)<\infty$ so that for all
$f\in L^{p_d ,1}(\bbR^d)$ and $b\in \bbR^{N+1}$,
\begin{equation} \label{affrestr-2}
\Big(\int_{-\infty}^\infty
 |\widehat f (\Gamma_b(t))|^{p_d} \, W_b(t) \, dt\Big)^{1/p_d} \le
C(N)\|f\|_{L^{p_d ,1}(\bbR^d)}.
\end{equation}
\end{theorem}
Both Theorems \ref{powerthm} and \ref{polthm} are optimal with respect to the two Lorentz exponents occurring on both sides, if we consider them as weighted Lorentz norm estimates: $L^{p_d,1}(\bbR^d) \to L^{p_d, p_d} (w \, dt)$. In particular, the strong type $(L^{p_d}, L^{p_d})$ estimate fails.
This fact is an easy consequence of the corresponding result in \cite{BOS1} for the nondegenerate case, where it was shown that $L^{p_d,1}(\bbR^d)$ was the smallest possible space and $L^{p_d,p_d}(w \, dt)$ the largest possible space on the scale of Lorentz spaces.
Moreover, the weight functions $w$ ($=w_a$ or $W_b$) are sharp up to a multiplicative constant. (See \cite{BOS3, Obaff} and Section 2 below.)

\medskip

\begin{remark}
One can also consider general polynomial curves of the form $\gamma(t) = (P_1(t), \cdots, P_d(t))$, where each $P_j$ is a polynomial of degree at most $N$. Dendrinos and Wright \cite{DeW} established the uniform Jacobian estimate for the mapping $(t_1, \cdots, t_d) \mapsto \sum_{j=1}^d \gamma(t_j)$. This implies a uniform restriction estimate in the reduced range $1\le p < p_c(d) = \frac{d^2+2d}{d^2+2d-2}$. (This range is commonly referred to as `Christ's range' of exponents.) This is the range where one does not need the `method of offspring curves', hence the torsion bound is not needed here. In \cite{BOS3} (see Proposition 8.1 there) this range was extended a little by combining an argument of Drury \cite{D2} with a result of Stovall \cite{Sto} on averaging operators.

The main obstacle for obtaining a uniform estimate in the full range, by means of the method of offspring curves, is that the second crucial estimate concerning the {\it torsion} of the offspring curves (as described in the beginning of Section 6) breaks down for curves of non-simple type. At the moment the only known approach that gives the full range $1\le p<p_d$ (and also the restricted strong type for $p=p_d$) for curves of non-simple type is the method based on `exponential parametrization', which originated in \cite{DM2} and was used in \cite{BOS3} to prove Theorem \ref{powerthm}. (See also \cite{DeM} and the remark at the end of Section 6 of \cite{BOS3}.)
\end{remark}

{\sl Notation.} Adopting the usual convention, we let $C$ or $c$ represent strictly positive constants whose value may not be the same at each occurrence. These constants may usually depend on $N$, $d$ and $p$, but they will always be independent of $f$. (In addition, they are uniform over the class of $\gamma(z)$ given in Theorem \ref{comppolthm}. In particular, they are independent of the coefficients of the polynomial $\phi(z)$ throughout the proof of that result.) Their dependence on the parameters is sometimes indicated by a subscript or shown in parentheses. We write $A \lc B$ or $B \gc A$ to mean $A \le C B$, and $A\approx B$ means both $A\lc B$ and $B\lc A$.

\medskip
{\sl Complex curves.}
Let us now consider an analogous problem for a `complex curve' in
$\bbC^d$, $d\ge 2$, of simple type. By this we mean a mapping of the following form:
\begin{equation} \label{compcurve}
z\mapsto \gamma (z) = \big(z, {z^2},\cdots, {z^{d-1}}
, \phi (z)\big), \quad z\in \Omega
\end{equation}
where  $\phi(z)$ is an analytic function on a domain $\Omega \subset \bbC$. We will regard this mapping as a 2-dimensional surface in $\bbR^{2d}$, given by the real mapping
\[ z = (x,y)\mapsto \gamma(x,y) = \big( x,y, {x^2-y^2}, 2xy, \cdots, \Re(\phi(z)), \Im(\phi(z)) \big)
.\]
In what follows we use $\bbC$ and
$\bbR^2$ interchangeably when there is no danger of confusion.

In analogy with the real case let us define a weight function by
\begin{equation} \label{compwt}
w(z) = |\tau (z)|^{4/(d^2+d)}, \text{ where } \tau(z) = \det (\gamma '(z), \cdots, \gamma^{(d)}(z)) .
\end{equation}
For $\gamma$ given by \eqref{compcurve}, we have $\tau(z) = c_d \, \phi^{(d)}(z)$ with $c_d = 2! \cdots (d-1)!$.
Let $d\mu$ denote the surface measure given by $d\mu(z) = d\mu (\gamma(z)) =dxdy$ for $z=x+iy$. The expression
$w(z) \, d\mu(z) = |\tau (z)|^{4/(d^2+d)} d\mu(z)$
is an analogue of the affine arclength measure for real curves (see \eqref{theweight}; {\it cf.} \cite{DM1,DM2,BOS1}). See Section 2 for the optimality of this choice of measure.

\medskip

When $d=2$, Oberlin \cite{Ob} proved the following
\begin{theorem}{\rm (\cite{Ob}; Theorem 4 and Example 3)} \label{comprestthm-2}
Let $\gamma(z)= (z, \phi(z))$, where $\phi(z)$ is an analytic
function on an open set $D\subset \bbC$. Suppose that $\phi'(z)$
and the map $(z_1, z_2)\mapsto (z_1 - z_2, \phi(z_1) - \phi(z_2))$
both have generic multiplicities at most $N$ on $D$ and $D^2$,
respectively.\footnote{Recall that $F : D \subset \Real^k \rightarrow \Real^k$ is said to have \emph{generic multiplicity $N$} if $\textrm{card}[F^{-1}(y)] \leq N$ for almost all $y \in \Real^k$. Here, $\textrm{card}[E]$ denotes the cardinality of the set $E$.}
Then there is a constant $C_p(N)<\infty$ so that for all $f\in
L^{p}(\bbR^{4})$,
\begin{equation} \label{comp-2}
\Big(\int_{D}
 |\widehat f (\gamma (z) )|^q \,|\phi''(z)|^{2/3} d\mu(z)\Big)^{1/q} \le
C_p (N)\|f\|_{L^{p}(\bbR^4)}
\end{equation}
whenever $1/p + 1/(3q) = 1$, $1\le p<4/3$.
\end{theorem}

Here, $\widehat f (\gamma (z) )$ stands for $\widehat f (\gamma (x,y) )$.
See \cite{Ch} for a related result for some 2-dimensional surfaces in $\bbR^{4}$ which are not necessarily given by holomorphic functions, but which satisfy a certain nondegeneracy condition. (See also \cite{DG} for an analogous result for some $k$-dimensional surfaces in $\bbR^d$, where $d=2k$.)

In this paper we obtain some positive results in higher dimensions.
First let us assume that
$\gamma(z)$ is in the form \eqref{compcurve}, where $\phi(z) = z^N$,
$z\in \bbC$, for an integer $N \geq 0$.
%

%
\begin{theorem} \label{compthm}
Given integers $d\ge 3$ and $N\ge 0$, let $\gamma(z)$ be as in \eqref{compcurve}, with $\phi(z)=z^N$. Then there is a constant $C(N)<\infty$ so that for all $f\in L^{p_d
,1}(\bbR^{2d})$,
\begin{equation} \label{affrestr-d}
\Big(\int_{\bbR^2}
 |\widehat f (\gamma(z) )|^{p_d} \,w(z) \, d\mu(z)\Big)^{1/p_d} \le
C(N)\|f\|_{L^{p_d ,1}(\bbR^{2d})}
\end{equation}
where $w(z) = |\phi^{(d)}(z)|^{4/(d^2+d)}$ and $p_d = (d^2+d+2)/(d^2+d)$.

Moreover, there is a constant $C_p (N) <\infty$ such that
\begin{equation} \label{affres-dpq}
\Big(\int_{\bbR^2}
 |\widehat f (\gamma(z) )|^{q} \,w(z) \, d\mu(z)\Big)^{1/q} \le
C_p (N) \|f\|_{L^{p}(\bbR^{2d})}
\end{equation}
whenever $1/p+ 2/[(d^2+d)q] = 1$, $1\le p< p_d$.
\end{theorem}
These estimates (as well as those in the following theorem) are expected to be optimal on the Lorentz scale of exponents, in view of the analogous results in the real case (see \cite{BOS1}). However, this seems to be difficult to show in the present context, where the (real) dimension of the surface $k$ is 2. For instance, it is unknown if the estimate \eqref{T-pq} below, which is dual to \eqref{affres-dpq}, fails for $q \le q_d$, $d\ge 3$, even when $f$ is a bump function and we are in the nondegenerate case (with $w=1$). This is related to the unsolved problem of determining the convergence exponent for the multi-dimensional Tarry's problem. In this connection, compare the statements of Theorem 1.3 (for $k=1$) and Theorem 1.9 (for $k\ge 2$) in \cite{ACK}. Notice that no information is available for the divergence of the integral in Theorem 1.9, while Theorem 1.3 gives the complete answer in the 1-dimensional case.

We show the sharpness of the condition $1/p+ 2/[(d^2+d)q] = 1$ at the end of this section (see under the heading ``A homogeneity argument"), and we also prove in section 2 the optimality of the weight function $w(z)$, given after \eqref{affrestr-d}.
\\

When $d=3$, we get an exact analogue of Theorem \ref{polthm}, valid for an {\it arbitrary} polynomial $\phi(z)$ of degree at most $N$.\footnote{It will be interesting if one can show a version of Theorem \ref{comppolthm} for higher dimensions ($d\ge 4$) as well as an analogue of Theorem \ref{powerthm} for complex curves.}
\begin{theorem} \label{comppolthm}
For $d = 3$ and $N\ge 0$, let $\gamma(z)=(z, z^2, \phi(z))$, where $\phi(z)$ is an arbitrary polynomial of degree at most $N$. Then there is a constant $C(N)<\infty$, independent of the coefficients of $\phi(z)$, so that for all $f\in L^{7/6,1}(\bbR^{6})$,
\begin{equation} \label{affrestr-d3}
\Big(\int_{\bbR^2}
 |\widehat f (\gamma(z) )|^{7/6} \,w(z) \, d\mu(z)\Big)^{6/7} \le
C(N)\|f\|_{L^{7/6 ,1}(\bbR^{6})}
\end{equation}
where $w(z) = |\phi''' (z)|^{1/3}$.

Moreover, there is a constant $C_p (N) <\infty$, independent of the coefficients of $\phi(z)$, such that
\begin{equation*} \label{affrestr-dpq}
\Big(\int_{\bbR^2}
 |\widehat f (\gamma(z) )|^{q} \,w(z) \, d\mu(z)\Big)^{1/q} \le
C_p (N) \|f\|_{L^{p}(\bbR^{6})}
\end{equation*}
whenever $1/p+ 1/(6q) = 1$, $1\le p< p_3 = 7/6$.
\end{theorem}

One can show that the weight functions $w(z)$ in \eqref{affrestr-d} and \eqref{affrestr-d3} are sharp up to a multiplicative constant, as in the real case.
See Proposition \ref{sharpweight} below.

\medskip

{\sl The dual estimate.}
Let $p'$ denote the H\"older conjugate exponent, i.e. $1/p+1/p'=1$. The dual estimate of \eqref{affrestr-d} is the following weak type $(q_d, \,q_d)$
estimate for $q_d = p_d' = (d^2+d+2)/2$:
\begin{equation}\label{T-Q}
\| T f\|_{L^{q_d,\infty}(\bbR^{2d})} \le C(N) \| f\|_{L^{q_d}(w d\mu )}
\end{equation}
where $T$ is given by
%
\begin{equation*}\label{T}
T f(x)=  \int_{\bbR^2} e^{i x\cdot \gamma(z)} f(z) w(z) d\mu(z),
\quad x\in \bbR^{2d} .
\end{equation*}
Recall that the mapping $z\mapsto \gamma(z)$ is regarded as a
2-dimensional surface $(x,y)\mapsto \gamma(x,y)$ in $\bbR^{2d}$. In
particular, $x\cdot \gamma(z)$ denotes the dot product in
$\bbR^{2d}$.

By interpolating \eqref{T-Q} with the $(L^1, L^{\infty})$
estimate it follows that
\begin{equation}\label{T-pq}
\| T f\|_{L^q (\bbR^{2d})} \le C_{q} (N) \| f\|_{L^p(w d\mu )}
\end{equation}
for $1/p+ (d^2+d)/(2q)=1$, $q> q_d = p_d' = (d^2+d+2)/2$.

\medskip
{\sl A homogeneity argument.}
To see the necessity of the condition $1/p+(d^2 + d)/(2q) = 1$ for \eqref{T-pq}
or \eqref{T-Q} to hold, we use the usual homogeneity argument. That is, we take $f= \chi_{B_R}$, where
$B_R = B(0,R)$ is a ball in $\bbR^2$.
We see that
\[ |Tf(x)| \gc R^{\frac{4(N-d)}{(d^2+d)}+2} \chi_{E_R} (x/a)
\]
for some small constant $a>0$, where $E_R = [-R^{-1}, R^{-1}]^2
\times [-R^{-2}, R^{-2}]^2 \times \cdots \times [-R^{-(d-1)}, R^{-(d-1)}]^2 \times [-R^{-N}, R^{-N}]^2$.
Hence, if \eqref{T-Q} or \eqref{T-pq} holds, then we must have
\[ R^{\frac{4(N-d)}{d^2+d}+2} R^{-\frac{2}{q}(\frac{d(d-1)}{2} +N)} \lc R^{(\frac{4(N-d)}{d^2+d}+2) \frac{1}{p} }, \quad \forall R>0 .\]
Thus, it follows that $1/p+(d^2 + d)/(2q) = 1$.

\medskip
{\sl Organization of this paper.} The optimality of the weight function $w(z)$ in Theorem \ref{compthm} or Theorem \ref{comppolthm} is proved in Section 2. Section 3 contains the proof of a lower bound for a Jacobian arising in the proof of Theorem \ref{compthm}. A uniform lower bound for the Jacobian associated to curves of simple type with arbitrary polynomials $\phi(z)$ is proved in Section 4. There is also a short discussion about a sublevel set estimate for the complex Vandermonde determinant at the end of Section 4. In Section 5 we state an interpolation theorem proved in \cite{BOS3}. Theorem \ref{comppolthm} is proved in Section 6.
Finally, in Section 7 we indicate how to modify the latter argument to prove
Theorem \ref{compthm}.


\section{Optimality of the weight function}

Let $d\ge 2$. Here we shall consider the more general mapping $\gamma(z) = (\phi_1(z), \cdots, \phi_d(z))$, where each $\phi_j$ is an analytic function on $\Omega \subset \bbC$. We continue to use the notation $\tau(z) = \det (\gamma '(z), \cdots, \gamma^{(d)}(z))$. The following result is analogous to one found in section 2 of \cite{BOS3}, which in turn is based on an argument in \cite{Obaff}.
\begin{proposition}\label{sharpweight} Assume that for some $p \in (1,\, p_d]$ and $q(p) = 2p'/(d^2+d)$ there is a constant $B$ such that for all $f\in L^{p,1}(\bbR^{2d})$,
\begin{equation} \label{hypo}
\Big(\int_{\Omega}
 |\widehat f (\gamma(z) )|^{q(p)} \, \omega(z) \, d\mu(z)\Big)^{1/{q(p)}} \le
B \|f\|_{L^{p,1}(\bbR^{2d})}
\end{equation}
where $\omega(z)$ is a nonnegative, locally integrable weight function on $\Omega$.
Then there is a constant $C_d$ such that
\begin{equation} \label{omega}
\omega(z) \le C_d \, B^{q(p)} |\tau(z)|^{4\over d^2+d}\quad a.e. ~ z \in \Omega .
\end{equation}
\end{proposition}
When $\gamma(z)$ is as in \eqref{compcurve}, then we have $\tau(z) = c_d \, \phi^{(d)}(z)$, so that the last inequality becomes $\omega(z) \le C_d \, B^{q(p)} |\phi^{(d)}(z)|^{4/(d^2+d)}$, as we wanted to show.

\begin{proof}
Let $P = A Q + b$ be a parallelepiped in $\bbR^{2d}$, where $Q=[-{1\over 2}, {1\over 2}]^{2d}$,
$b\in \bbR^{2d}$ and $A$ is an invertible linear transformation on $\bbR^{2d}$.
Take $\widehat{f}(\xi) = \exp (- \pi|A^{-1} (\xi - b)|^2 )$. Then $|\widehat{f}(\xi)| \ge c_0 >0$ for $\xi \in P$, and $f(x) = e^{2\pi i b\cdot x} |\det(A)| \cdot \exp (- \pi |A^t \, x|^2)$. Since $|P| = |\det(A)|$, we have $\| f\|_{p,1} \approx |P|^{1/p'}$. Hence,
\eqref{hypo} implies that
\begin{equation} \label{paraest}
\int_{\bbR^2}
 \chi_P (\gamma(z) )\, \omega(z) \, d\mu(z) \le C(d) \,
B^{q(p)} |P|^{2/(d^2+d)} .
\end{equation}

Since each $\phi_j (z)$ is analytic on $\Omega$, so is $\tau(z)$.
Thus, we may assume $\tau(z)$ has only isolated zeros. So, it is enough to show \eqref{omega} at points where $\tau(z) \not= 0$. (Otherwise, $\tau(z)$ is identically zero. We comment on this case at the end of this section.)

Fix $a\in \Omega$. We have
\begin{equation} \label{taylorexp}
\gamma(a+z) = \gamma(a) + \sum_{j=1}^d {z^j\over j!} \gamma^{(j)}(a) + O(|z|^{d+1})
\end{equation}
for $z$ near the origin. Now consider the linear mapping
\begin{equation} \label{mapping}
(z_1, \cdots, z_d) \mapsto \Phi(z_1, \cdots, z_d) = \gamma(a) + \sum_{j=1}^d {z_j\over j!} \gamma^{(j)}(a) .
\end{equation}
Write $z_j = x_j + i y_j$. For $\eps >0$, let $E = \{ (z_1, \cdots, z_d):~ |x_j| \le 2 \, \eps^j,~  |y_j| \le 2 \, \eps^j, ~ 1\le j\le d \}$ denote a rectangular box in $\bbR^{2d}$. The image $P_1$ of $E$ under this mapping is a parallelepiped in $\bbR^{2d}$. Its volume $|P_1|$
is given by
\begin{align*}
|P_1| &= 2^{2d} \, \eps^{d^2+d} \cdot J_{\bbR}\Phi = 2^{2d} \, \eps^{d^2+d} \cdot |\det J_{\bbC}\Phi|^2\\
&= 2^{2d} \,\eps^{d^2+d} \cdot |(2! \cdots d!)^{-1} \det (\gamma'(a), \cdots, \gamma^{(d)}(a) )|^2 \\
&= 2^{2d} (2! \cdots d!)^{-2}\,
\eps^{d^2+d} \cdot |\tau (a)|^2 .
\end{align*}
We used here the fact that the Jacobian of \eqref{mapping} as a real mapping is given by $J_\bbR \Phi = |\det J_{\bbC} \Phi|^2$, where $J_{\bbC}\Phi$ is the holomorphic Jacobian matrix of the mapping \eqref{mapping}. This is a consequence of Proposition 1.4.10 on p. 51 in \cite{Kr}.

If $\tau(a) \not= 0$, and if $\eps = \eps(a) >0$ is sufficiently small, then we have $\gamma(a+z) \in P_1$ when $|z| \le \eps$. In fact, since $\gamma'(a), \cdots, \gamma^{(d)}(a)$ span $\bbC^d$, it follows from \eqref{taylorexp} that
\begin{equation} \label{taylorexp-1}
\gamma(a+z) = \gamma(a) + \sum_{j=1}^d {z^j + z^{d} g_j (z,a)\over j!} \gamma^{(j)}(a)
\end{equation}
for some functions $g_j(z,a)$ such that $g_j(z,a) \rightarrow 0$ as $z\rightarrow 0$ for $j=1, 2, \cdots, d$.

Therefore, it follows from \eqref{paraest} that
\[ \limsup_{\eps \to 0} {1\over \pi \eps^{2}} \int_{|z| \le \eps} \omega(a+z) \, d\mu (z) \le C_d \, B^{q(p)} |\tau(a)|^{4/(d^2+d)} .
\]
%
So the conclusion \eqref{omega} follows by the Lebesgue differentiation theorem.

On the other hand, when $\tau(a) = 0$, a slight modification of the above argument shows that
\[ \int_{|z| \le \eps} \omega(a+z) \, d\mu (z) = o(\eps^2), \text{ as } \eps \rightarrow 0 .
\]
Thus, when $\tau(z) \equiv 0$, we may conclude that $\omega(z)$ is zero almost everywhere. (See section 2 of \cite{BOS3} for more details.)
\end{proof}

\section{A lower bound for the Jacobian}
%

Let us first set up the notation.
\begin{defn}\label{PN}
Let $N$ be a nonnegative integer and let $z_1, \cdots , z_d$ be complex numbers. Let $P_N$ denote a homogeneous monic polynomial of degree N in $z_1, \cdots , z_d$, given by
\[  P_N(z_1, \cdots, z_d) =
\sum_{\alpha_1 + \cdots +\alpha_{d} = N} z_1^{\alpha_1} \cdots z_d^{\alpha_d} .\]
Here, $\alpha_1, \cdots, \alpha_{d}$ are nonnegative integers.
\end{defn}

Thus, $P_N$ is a {\it symmetric} polynomial. We have the following properties of $P_N$:
\begin{lem}\label{rmk}
Let $d\ge 2$ and $N\ge 1$. Then

\noindent
$(i)$ $P_0 (z_d, \cdots, z_1) = 1$;

\noindent
$(ii)$ $P_N(z_3, z_1) - P_N(z_2, z_1) = (z_3 - z_2)P_{N-1}(z_3, z_2, z_1)$;

\noindent
$(iii)$ $P_N(z_d, z_{d-1}, \cdots, z_1) = P_N(z_d, \cdots, z_2) + P_{N-1}(z_d, \cdots , z_2) z_1 + \cdots +$
\\
$\text{ }  + P_1(z_d, \cdots, z_2) z_1^{N-1} + z_1^N$.

\noindent
$(iv)$ Moreover, we have
%
\begin{align*}
P_N(z_{d+1}, \, & z_{d-1}, \cdots, z_1) - P_N (z_d, z_{d-1}, \cdots, z_1) = \\
&= (z_{d+1} - z_d) \, P_{N-1}(z_{d+1}, \cdots,  z_1) .
\end{align*}
\end{lem}

\begin{proof}
The properties $(i)$-$(iii)$ are straightforward. To see that $(iv)$ holds, we use induction on $d$. First, $(ii)$ gives the case $d=2$.
Now suppose that $(iv)$ holds with $d$ replaced by $d-1$. That is, we assume
\[ P_N(z_d, z_{d-2}, \cdots, z_1) - P_N(z_{d-1}, z_{d-2},  \cdots, z_1) = (z_d - z_{d-1})P_{N-1}(z_d, \cdots , z_1)
\]
holds for some $d\ge 3$ and for $N\ge 1$. It follows from $(iii)$ and this induction hypothesis that
\begin{align*}
P_N ( &z_{d+1}, z_{d-1}, \cdots, z_1) - P_N(z_d, z_{d-1}, \cdots, z_1) \\
&= P_N(z_{d+1}, z_{d-1}, \cdots, z_2) + P_{N-1}( z_{d+1}, z_{d-1}, \cdots, z_2)z_1 + \cdots + z_1^N \\
&\qquad - [P_{N}(z_d, z_{d-1}, \cdots, z_2) + P_{N-1}( z_d, z_{d-1}, \cdots, z_2)z_1 + \cdots + z_1^N] \\
&= (z_{d+1} - z_d) \big[ P_{N-1}(z_{d+1}, z_{d}, \cdots, z_2) + P_{N-2}( z_{d+1}, z_{d}, \cdots, z_2)z_1  \\
 &\qquad \quad+ \cdots + P_1(z_{d+1}, z_{d}, \cdots, z_2)z_1^{N-2} + z_1^{N-1} \big] \\
&= (z_{d+1} - z_d)\, P_{N-1}(z_{d+1}, \cdots, z_1)
\end{align*}
which is the case $d$ of $(iv)$. Hence, $(iv)$ holds for all $d\ge 2$ and $N\ge 1$.
%
%
\end{proof}

We now turn to the proof of a lower bound for the Jacobian of a transformation that arises in the proof of Theorem \ref{compthm}.
Let $J(z, h_2, \cdots, h_d)$ denote the determinant of the holomorphic Jacobian matrix of the mapping $(z, h_2, \cdots , h_d) \mapsto \Gamma(z,h_2, \cdots,h_d) = \sum_{k=1}^d \Gamma_b(z+h_k)$ with $h_1=0$. Here, $\Gamma_b (z) = m^{-1 }\sum_{j=1}^m \gamma(z+
b_j)$, where $m\in \bbN$, and $b= (b_1, \cdots, b_m) \in \bbC^{m}$, with $b_1=0$.

\begin{lem}\label{jacest}
Let $\gamma(z)$ be given by \eqref{compcurve} with $\phi(z) = z^N$ for an integer $N \geq d$ with $d \ge 2$. Set $J(z, h_2, \cdots, h_d) = \det (\Gamma'_b(z+h_1), \cdots, \Gamma'_b(z+h_d))$, where $z, h_2, \cdots, h_d \in \mathbb{C}$ and $h_1 = 0$. Then $\mathbb{C}$ is the union of $C(d,N)$ sectors $\Delta_\ell$ centered at the origin such that for each $1 \leq \ell \leq C(d,N)$, and for each integer $m \geq 1$, we have
\begin{align}\label{J-max}  |J(z, h_2, \cdots, h_d)| \geq c(d,N) \, v(h)\, \max \Big\{ \frac{1}{m} \sum_{j=1}^m |\phi^{(d)}(z+b_j+h_k)| : 1 \leq k \leq d \Big\}
\end{align}
where $z+b_j+h_k \in \Delta_l$, and $v(h_2, \cdots, h_d) =|V(z_1, \cdots, z_d)|$ is the absolute value of the complex Vandermonde determinant, where we put $h_k = z_k-z_1 = z_k- z$. Here, $C(d,N)$ and $c(d,N)$ are positive constants depending only on $d$ and $N$.
\end{lem}

\begin{proof}
Let us write $z_{jk} = z+b_j + h_k$. Recall that $h_1=0$, and so $z_{j1} = z+b_{j}$. If we abbreviate $\sum_{j=1}^m$ as $\sum$, we get
\begin{align*}
&J(z, h_2, h_3, \cdots, h_d) = \det (\Gamma'_b(z+h_1), \cdots, \Gamma'_b(z+h_d)) \\
& = \frac{(d-1)!N}{m^{d-1}}
\left|
\begin{array}{cccc}
 1 & 1  & \cdots & 1\\
\sum (\zbh1)  & \sum (z_{j2})  & \cdots & \sum (z_{jd})\\
\sum(\zbh1)^2  & \sum(z_{j2})^2  & \cdots & \sum (z_{jd})^2\\
\vdots & \vdots & \ddots & \vdots \\
\sum(\zbh1)^{d-2}  & \sum(z_{j2})^{d-2}  & \cdots & \sum (z_{jd})^{d-2}\\
\sum(\zbh1)^{N-1}  & \sum(z_{j2})^{N-1}  & \cdots & \sum (z_{jd})^{N-1}\\
\end{array}
\right|
\end{align*}
\begin{align*}
 = \frac{(d-1)!N}{m^{d-1}}
\left|
\begin{array}{cccc}
1  &  0  & \cdots & 0 \\
\sum (\zbh1) & mh_2 & \cdots & m h_d \\
\sum(\zbh1)^2 & h_2 \sum P_1(\zbh{2}, \zbh1) & \cdots & h_d \sum P_1(\zbh{d}, \zbh1) \\
\vdots & \vdots  & \ddots & \vdots \\
\sum (\zbh1)^{d-2} & h_2 \sum P_{d-3}(\zbh{2}, \zbh1) & \cdots & h_d \sum P_{d-3}(\zbh{d},\zbh1) \\
\sum(\zbh1)^{N-1} & h_2 \sum P_{N-2}(\zbh{2}, \zbh1) & \cdots & h_d \sum P_{N-2} (\zbh{d}, \zbh1)
\end{array}
\right| .
\end{align*}
%
Note that the value of this determinant equals
\begin{align*}
&\left|
\begin{array}{cccc}
m h_2  &  m h_3  & \cdots & m h_d \\
h_2 \sum P_1(\zbh{2},\zbh1) & h_3 \sum P_1(\zbh{3},\zbh1)  & \cdots & h_d \sum P_1(\zbh{d},\zbh1) \\
h_2 \sum P_2(\zbh{2},\zbh1) &  h_3 \sum P_2(\zbh{3}, \zbh1) & \cdots & h_d \sum P_2(\zbh{d},\zbh1) \\
\vdots & \vdots  & \ddots & \vdots \\
h_2 \sum P_{d-3}(\zbh{2},\zbh1) & h_3 \sum P_{d-3}(\zbh{3}, \zbh1) & \cdots & h_d \sum P_{d-3}(\zbh{d},\zbh1) \\
h_2 \sum P_{N-2}(\zbh{2},\zbh1) & h_3 \sum P_{N-2}(\zbh{3}, \zbh1) & \cdots &  h_d \sum P_{N-2} (\zbh{d}, \zbh1)
\end{array}
\right|
\end{align*}
\begin{align*}
= &\,m h_2 h_3 \cdots h_d \times \\
& \times\left|
\begin{array}{ccc}
1  &  0  & \cdots \\ 
\sum P_1(z_{j2} ,\zbh1) & m(h_3 - h_2)  & \cdots \\ 
\sum P_2(z_{j2},\zbh1) &  (h_3 - h_2) \sum P_1(z_{j3}, z_{j2}, \zbh1) & \cdots \\ 
\vdots & \vdots  & \ddots \\ 
\sum P_{d-3}(z_{j2},\zbh1) & (h_3 - h_2)\sum P_{d-4}(z_{j3},z_{j2}, \zbh1) & \cdots \\
\sum P_{N-2}(z_{j2},\zbh1) & (h_3 - h_2)\sum P_{N-3}(z_{j3},z_{j2}, \zbh1) & \cdots \\
\end{array} \right. \\
& \qquad \qquad \qquad \qquad \qquad \qquad \left. \begin{array}{cc}
\cdots & 0 \\
\cdots & m(h_d - h_2) \\
\cdots & (h_d - h_2)\sum P_1(z_{jd},z_{j2},\zbh1) \\
\ddots & \vdots \\
\cdots & (h_d - h_2)\sum P_{d-4}(z_{jd},z_{j2}, \zbh1) \\
\cdots & (h_d - h_2)\sum P_{N-3} (z_{jd}, z_{j2}, \zbh1)
\end{array} \right|
\end{align*}
by the properties of $P_N$ stated in Lemma \ref{rmk}.

Continuing in this way, we see that
\begin{align*}
& J(z, h_2, \cdots, h_d) \\
& = (d-1)!N m^{-1} (h_2 \cdots h_d)\cdots(h_{d-1} - h_{d-2})(h_d - h_{d-2}) \times \\
&\left|
\begin{array}{cc}
1  & 1 \\
\sum P_{N-d+1} (z_{j, d-1}, z_{j, d-2} , \cdots , \z_{j1}) & \sum P_{N-d+1} (z_{j,d} , z_{j, d-2} , \cdots , \z_{j1})
\end{array}\right| \\
&= \frac{(d-1)!N}{m} \prod_{1 \leq k < l \leq d} (h_l - h_k) \sum_{j=1}^m P_{N-d} (z_{j,d} , z_{j,d-1} , \cdots , \z_{j1}).
\end{align*}
Hence, if we write $L_j$ for $P_{N-d}(z_{jd}, \cdots , \z_{j1})$, we obtain
\[  |J(z, h_2, \cdots , h_d) |  \geq \frac{(d-1)!N}{m} v(h)\cdot \Big|\sum_{j=1}^m L_j \Big|.
\]

By rotation, it suffices to consider the case $\Delta_\ell = \Delta = \{ z= x+i y \in \mathbb{C} : 0 < y < \varepsilon x \}$ with some small $\varepsilon = \varepsilon(d,N) > 0$. (Indeed, we may express the elements of $\Delta_\ell$ in the form $z' = a z$, for $z\in \Delta$ and some fixed complex number $a$ with $|a|=1$. By homogeneity, the powers of $a$ may be factored out of each row of the Jacobian.)

Recalling that $z_{jk} = z+b_j + h_k$, let us write $x_{jk} = \textrm{Re} (z_{jk})$ and $y_{jk} = \textrm{Im} (z_{jk})$. Then for each $j$, we have the lower bound
\[  |\textrm{Re} [L_j]| \geq P_{N-d}(x_{j1}, x_{j2} , \cdots, x_{jd}) + E_j  \]
where $E_j$ is a sum of $C(d,N)$ terms similar to the expression preceding it but with one or more factors $x_{jk}$ replaced by $c_{jk}\, y_{jk}$. Here, $|c_{jk}| \le C'(d,N)$. Recall that $0 < y_{jk} < \varepsilon x_{jk}$. Hence the last expression is bounded below by
\[  \frac{1}{2} P_{N-d} (x_{j1}, x_{j2} , \cdots, x_{jd}) \gtrsim \sum_{k=1}^d x_{jk}^{N-d} \approx \sum_{k=1}^d |\phi^{(d)}(z+b_j + h_k)|
\]
provided that $\varepsilon = \varepsilon(d,N) >0$ is chosen sufficiently small. This implies that
\[   | J(z, h_2, \cdots, h_d) |   \geq c(d,N) v(h) \frac{1}{m} \sum_{k=1}^d \sum_{j=1}^m |\phi^{(d)} (z + b_j + h_k)| \]
whenever $z+b_j + h_k \in \Delta$. This finishes the proof.
\end{proof}


\section{Jacobian bound for polynomial curves of simple type in $\bbC^3$}

A version of the following lemma may be found in \cite{FW} (Lemma 3.1), where it is stated and proved for polynomials of a real variable. (See also \cite{CRW, CRW2}.) But the same proof works for polynomials of a complex variable, since it only relies on the triangle inequality.
\begin{lem}\label{Gap-dyadic}
Given a complex number $D\not=0$, let $P(z) = D \prod_{j=1}^N (z-z_j) = \sum_{k=0}^N \nu_k \,z^k$ be a polynomial of degree $N$. Assume that the roots $z_j$ are ordered so that $|z_1| \le \cdots \le |z_N|$. Let $G_j =\{ z\in \bbC: ~ A|z_j| \le |z| \le A^{-1} |z_{j+1}|\}$ for $1\le j\le N-1$, and $G_N = \{ z\in \bbC: ~ |z| \ge A|z_N|\}$. Then there exists a constant $C=C(N)>1$ such that for any $A\ge C(N)$ and $1\le j\le N$, if $G_j$ is nonempty, then
\\
$(i)$ $|P(z)| \approx |\nu_j| |z|^j$ for $z \in G_j$ ;
\\
$(ii)$ for $1\le j\le N-1$, we have $|\nu_j| \approx |D| \prod_{\ell = j+1}^N |z_\ell|$. (For $j=N$, we have $\nu_N = D$. In particular, $\nu_j \not=0$, $1\le j\le N$.)
%
\end{lem}
%






The idea of this lemma helps us prove a uniform lower bound for the Jacobian associated to complex curves of simple type in $\bbC^3$, when $\phi(z)$ is an arbitrary polynomial.
This result may be of some independent interest. For instance, it is likely to have some implications for the related averaging operators.
(See e.g. \cite{DLW}, \cite{Sto}.)
\begin{lem}\label{J-simplepoly}
Let $\gamma(z) = (z, z^2,\cdots, z^{d-1}, \phi(z))$, where $\phi(z)$ is a polynomial of degree at most $N$.
%
Let $J (u_1, \cdots , u_d) = J_\bbC (u_1, \cdots , u_d)$ be the determinant of the holomorphic Jacobian of the transformation $(u_1, \cdots , u_d) \mapsto \sum_{i=1}^d \gamma(u_j)$.

If $d=3$, then there exist a constant $c(d, N) > 0$, a positive integer $M= M(d , N)$, and a collection of pairwise disjoint, convex open sets $B_1, \cdots , B_M$,
such that
$\bbC = \cup_{\ell =1}^M B_\ell$, ignoring a null-set,
and such that for $1\le \ell \le M$,
\begin{equation}\label{Jmax}
|J (u_1, \cdots , u_d)| \ge c(N)\, V(u_1, \cdots , u_d) \max_{1\le i\le d} |\phi^{(d)}(u_i)|
\end{equation}
whenever $u_j \in B_\ell$, $1\le j\le d$.
\end{lem}

Recall that $V(u_1, u_2, u_3) = |u_1 - u_2| \cdot |u_1 - u_3| \cdot |u_2 - u_3|$, when $d=3$.

\begin{remark}\label{rem-Gamma}
If $\gamma(z)$ in Lemma \ref{J-simplepoly} is replaced by
\[ \Gamma(z) = (P_1(z), \cdots, P_{d-1}(z), \phi(z))
\]
as in \eqref{offsp} below, then the Jacobian of the corresponding mapping is the same as for $\gamma(z)$ when they have the same $\phi(z)$. So, we should obtain the same conclusion \eqref{Jmax} in this case. For example, when $d=3$, the new Jacobian $J(u_1, u_2, u_3)$ is again given by the formula \eqref{Jrep} below.
\end{remark}

\begin{proof}[Proof of Lemma \ref{J-simplepoly}] Let $d=3$. If $0\le N\le 2$, then $\phi'''\equiv 0$ and $J\equiv 0$. Moreover, if $N=3$, then $\phi'''(z)$ is a non-zero constant and $J(u_1, u_2, u_3)$ is a constant multiple of $V(u_1, u_2, u_3)$. Thus, we may assume that $N\ge 4$ and $\phi'''(z)$ has at least one zero.
Our goal is to decompose $\bbC$ into a collection $\{ B\}$ of $M(N)$ pairwise disjoint, convex open sets so that the inequality \eqref{Jmax} holds on each $B$. To this end, we will represent $J(u_1, u_2, u_3)$ as an integral as in \eqref{Jrep} below.
It may be worthwhile to point out that, compared to the real case, the complex case is more delicate, because it is necessary to control carefully the argument of the integrand as well as the magnitude, in order to get a good lower bound for the multiple integral of a function of a complex variable.
\medskip

For the sake of clarity we will divide the rest of the proof into four steps.

\medskip

\noindent
{\sl Step 1. Preliminary decompositions of $\bbC$.}

To get a decomposition of $\bbC$, we begin by fixing a zero $b$ of $\phi'''(z)$. Let $P(z) = \phi'''(z+b)$. Then $P(0)=0$. Let us write $P(z)= D z^{a_1} \prod_{j=2}^m (z-\eta_j)^{a_j}$, where 0 and $\eta_j$ are the distinct roots of $P(z)$, with multiplicity $a_j$, so that $N - 3 = a_1+\cdots+ a_m$. Put $S_1 = S_1(b)= \{ z\in \bbC: ~ |z| < |z-\eta_j|, ~ \forall j \not= 1\}$ as in \cite{DeW}. We will decompose $S_1$ further in four different ways.

\medskip

{\sl (1.a) Decomposition into gap annuli and dyadic annuli.}
Let us rewrite $P(z) = D \prod_{j=1}^N (z-z_j) = \sum_{k=0}^N \nu_k \, z^k$ as in Lemma \ref{Gap-dyadic}, with $z_j$, $\nu_j$ and $G_j$ as in that lemma. (By abuse of notation we will write $N$, instead of $N-3$, for $\deg(P)$. Thus, we have $N\ge 1$ in this new notation.) Since a constant factor in $\phi(z)$ can be canceled from both sides of the inequality \eqref{Jmax}, we may assume that $D=1$. Since $P(0)=0$, we have $z_1=0$. The region $G_j$ may be called a `gap annulus' in analogy with the terminology `gap interval' in \cite{DeW}.
From Lemma \ref{Gap-dyadic} it follows that
$|P(z)| \approx |\nu_j| |z|^j$ for $z \in G_j$.
Also, define the `dyadic annuli' by
\[ D_j = \{ z\in \bbC: ~ A_1^{-1} |z_j| < |z| < A_1 |z_j|\}, ~ 2\le j\le N-1 ,
\]
for some $A_1 >0$ chosen slightly larger than $A$.
Thus, there is a small overlap between the regions $G_j$ and $D_j$, which will help us define certain {\it convex} open sets $B$ contained in them, cutting off some parts of the non-convex regions (annuli) $G_j$ and $D_j$. (See the second paragraph under the heading {\it `Decomposition of $G_j$'} below.)
\medskip

{\sl (1.b) Decomposition into sectors.}
By dividing $\bbC$ into narrow sectors $\{ \Delta\}$ centered at $0$ and then by using rotation, we may assume $0< y < \eps x$ in $\Delta$, for some $\eps = \eps(N)$, where we have written $z - b =x+iy$.
Then we have $|\phi'''(z)| = |P(z-b)| \approx |\nu_j| \cdot |z-b|^j \approx |\nu_j| \cdot x^j$, for $z-b \in \Delta \cap G_j$.
\medskip

{\sl (1.c) An integral representation of the Jacobian.}
Assume that $U$ is a convex open set. (We will take $U=b+B$ later.) Let $u$, $v$, $w \in U$.
Let $\theta$ be the largest of the interior angles of the triangle $uvw$. Then $\pi/3 \le \theta \le \pi$. By renaming the points if necessary, we may assume that the angle at $v$ equals $\theta$ and that $|v-u| \le |w-v|$.
We have the representation
\begin{equation}\label{Jrep}
J (u, v, w) = \int_u^v \int_v^w \int_{s_1}^{s_2} \phi'''(z) \, dz\, ds_2 \,ds_1 = \int_u^v \int_v^w \int_{s_1}^{s_2} P(z-b) \, dz\, ds_2 \,ds_1
\end{equation}
where each integral is regarded as a line integral over a line segment. (This is where we need the convexity of $U$.)

By factoring out a unit complex number, we may also assume that $v-u$ is a positive real number. This amounts to having the vector $\overrightarrow{uv}$ horizontal and pointing to the right. We parametrize the line integrals above by setting $s_1 = u + (v-u) t_1$, $s_2 = v + (w-v) t_2$, and $z = s_1 + (s_2 - s_1) t_3$, with $0 \le t_j \le 1$, to obtain
\[ J (u, v, w) = (v-u) (w-v) \int_0^1 \int_0^1 \int_0^1 [s_2(t_2) -s_1 (t_1)] \phi'''(z (t_1, t_2, t_3)) \, dt_3 \, dt_2 \, dt_1 .
\]
\medskip

\noindent
{\sl Step 2. Further decompositions of the regions.}

Recall that $\phi'''(z+b) = P(z) = \prod_{j=1}^N(z-z_j)$. Let us rewrite it in the form
\begin{equation}\label{gz}
 P(z) = g(z)(-1)^{N-j}\, z^j \prod_{\ell=j+1}^N z_\ell
\end{equation}
where
\[ g(z) = \prod_{i=1}^j \Big(1- {z_i\over z} \Big) \prod_{\ell =j+1}^N \Big(1- {z \over z_\ell} \Big).
 \]
%

We want to decompose the range of $g(z)$, contained in an annulus, into small radial sectors.
By considering the pre-images of the sectors we want to decompose $S_1 \cap \Delta \cap G_j$ and $S_1 \cap \Delta \cap D_j$ further into convex sets $\{ B\}$ with the following property:

After multiplying by a unit complex number if necessary, $g(z)$ can be put in the form $g(z) = \xi(z) + i \eta(z)$ with
\begin{equation}\label{b0ineq}
0 < b_0 \, |\eta(z)| \le \xi(z)
\end{equation}
for all $z \in B \subset S_1 \cap \Delta \cap E_j$ (with $E_j = G_j$ or $D_j$), where $b_0 >0$ is a large absolute constant to be chosen later. If this holds, then we have $\xi(z) \le |g(z)| \le (1+b_0^{-2})^{1/2} \xi(z)$ for $z\in B$.

\medskip
To achieve this goal, we need to decompose $G_j$ and $D_j$ further. This can be done separately for $G_j$ and $D_j$ as follows:

\medskip
{\sl (2.a) Decomposition of $G_j$.}
If $z\in S_1 \cap \Delta \cap G_j$, we have $A|z_j| \le |z| \le |z_{j+1}|/A$. We may assume $z_{j+1} \not=0$, since otherwise $G_j = \{ 0\}$ and there is nothing to prove.
%
%
%
Since $1- z_i/z= 1+ O(1/A)$, $1\le i\le j$, and also $1- z/z_\ell = 1+ O(1/A)$, $j+1\le \ell \le N$, taking $A= C_0 N$ gives $g(z) = 1+ O(C_0^{-1})$.
In fact, it is easy to see that $|g(z)-1| \le 2 \, C_0^{-1}$, which yields the condition \eqref{b0ineq} if we choose $C_0 \ge 3\, b_0$, say.

It only remains to cut $S_1 \cap \Delta \cap G_j$ into a few {\it convex} open sets $B$ so that their union covers all of $S_1 \cap \Delta \cap  G_j$, except for a null set and some little pieces which lie in the intersections $D_i \cap G_j \cap S_1 \cap \Delta$, for $i=j$ and $i= j+1$. (The remaining parts of the sets $D_i \cap G_j \cap S_1 \cap \Delta$, for $i=j,\, j+1$, will be covered by the $B$'s arising from the decomposition of $D_i$, which is described next.)


\medskip
{\sl (2.b) Decomposition of $D_j$.}
If $z \in S_1 \cap \Delta \cap D_j$, we have $A_1^{-1} |z_j| < |z| < A_1 |z_j|$, where $A_1 = (1+\delta_0 ) A = C_1 N = (1+\delta_0 )C_0 N$ for some small $\delta_0 >0$. We may assume $z_j \not= 0$ here, since otherwise $D_j$ is empty.
Let us recall $P(z) = g(z) (-1)^{N-j}\, z^j \prod_{\ell=j+1}^N z_\ell$, as in \eqref{gz}.

Note that $|(z-z_i)/z| \ge 1$ for all $i$ if $z\in S_1$, and also $|(z_\ell - z)/z_\ell| \ge (1/2)$ for all $\ell$ if $z\in S_1$. In fact, the second inequality follows from the first, since $|z_\ell| \le |z- z_\ell| + |z| \le 2 |z- z_\ell|$ if $z \in S_1$. From this it follows that
\begin{align}\label{g-lb}
|g(z)| \ge 2^{j-N} \ge 2^{2-N} \quad \forall ~ z \in S_1 \cap \Delta \cap D_j, ~ 2\le j\le N .
\end{align}
The inequality \eqref{g-lb} gives a separation from the origin, which is needed to obtain a small angular support for $g(B)$ so that \eqref{b0ineq} holds, where $B$ is to be specified shortly.

Moreover, we have
$|\partial_r (1- z_i/z)| \le |z_i|/r^2 \le |z_j|/r^2 \le A_1^2 / |z_j|$ (for $i\le j$) and $|\partial_r (1- z/z_\ell)| \le 1/|z_\ell| \le 1/|z_j|$ (for $\ell \ge j$).
Hence,
\[ |\partial_r (g(r,\theta))| \le N(1+A_1)^{N+1} |z_j|^{-1} .
\]
Likewise, we get $|\partial_\theta (1- z_i/z)| \le |z_i|/r \le |z_j|/r \le A_1$ (for $i\le j$) and $|\partial_\theta (1- z/z_\ell)| \le r/|z_\ell| \le r/|z_j| \le A_1$ (for $\ell \ge j$).
So, $|\partial_\theta (g(r,\theta)| \le N(1+A_1 )^N$.

Hence, we can divide the $r$-interval, given by $A_1^{-1} |z_j| < r < A_1 |z_j|$, into $C(N)$ pieces of length $L \le C(N)^{-1} A_1 |z_j|$ so that
\begin{align}\label{partial-r}
|\partial_r (g(r,\theta))| \cdot L &\le N (1+A_1)^{N+1} |z_j|^{-1} \times C(N)^{-1} A_1 |z_j| \\
&\le C(N)^{-1} N (1+A_1)^{N+2} . \notag
\end{align}
(Note that the two factors involving $|z_j|$ cancel out.)

Similarly, if we divide the $\theta$-interval into $C(N)$ pieces of angle $\Theta$, then we have
$|\partial_\theta (g(r,\theta))| \cdot \Theta \lc N (1+A_1)^N \times \eps(N)\, C(N)^{-1}$. Since this is smaller than the previous estimate, for simplicity we can use the same number $C(N)$ here.

This allows us to choose $C(N)^2$ pairwise disjoint, convex open sets $\{ B\}$ in $S_1 \cap \Delta \cap D_j$ such that $g(B)$ is contained in a small disk of diameter $\lc C(N)^{-1} N (1+A_1)^{N+2}$. We can do this in such a way that the collection $\{ B\}$, which consists of all the $B$'s from this step (for $D_j$, $2\le j\le N$) and the previous one (for $G_j$, $1\le j\le N$), covers all of $S_1\cap \Delta$, except for a null-set.

The estimates \eqref{g-lb} and \eqref{partial-r} imply that the angular support of $g(B)$ (when the angle is measured from $0$) is bounded by
\begin{align*}
 {C_2\, C(N)^{-1} N (1+A_1)^{N+2}\over c_0 \, 2^{2-N}} = {C_2 \, 2^N N (1+C_1 N)^{N+2}\over 4\, c_0 \, C(N)} < {1\over 2\, b_0} \notag
\end{align*}
if $C(N)$ is chosen so that $C(N) > b_0 \, c_0^{-1} C_2\, 2^{N-1} N (1+C_1 N)^{N+2}$. Therefore, we obtain \eqref{b0ineq} for every $z \in B \subset S_1 \cap \Delta \cap D_j$.
\medskip

\noindent
{\sl Step 3. A lower bound for the integral.}

Let us now put $s_2 - s_1 = s_2(t_2) - s_1 (t_1) = \alpha + i\beta$ and $H_j \cdot (z-b)^j = a + i\delta$, where $H_j = \prod_{k=j+1}^N |z_k|$. Thus, we have $\phi'''(z) = \pm (a + i\delta)(\xi + i\eta)$. By our assumptions, $\beta$ is single-signed.
Let us assume $\beta \ge 0$ for the sake of definiteness.
%
Since $|\delta| \le c \, \eps a$ when $z \in b + B \subset b+\Delta$, we have
\begin{align}\label{Repart}
\Re [(s_2-s_1)\phi'''(z)] &= (\alpha a - \beta \delta)\xi - (\beta a + \alpha \delta) \eta \\
&= \alpha a \xi - \beta a \eta + O(\eps |s_2-s_1| \, a \, \xi) ; \notag
\end{align}
\begin{align}\label{Impart}
\Im [(s_2-s_1)\phi'''(z)] &= (\alpha a - \beta \delta)\eta + (\beta a + \alpha \delta) \xi\\
&= \alpha a \eta + \beta a \xi + O(\eps |s_2-s_1| \, a \, \xi) . \notag
\end{align}
Note that the signs of \eqref{Repart} and \eqref{Impart}
do not affect our argument, because we estimate the absolute value of the Jacobian $J(u,v,w)$ from below as follows:
\[  |J (u, v, w)| \gc |v-u||w-v| \Big|\int_0^1 \int_0^1 \int_0^1 \Im[(s_2-s_1)\phi'''(z)] \, dt_3 dt_2 dt_1\Big| .
\]
%
%
%
Fix a set $B$ as above and assume that $u$, $v$, $w \in b+ B \subset b+ (S_1 \cap \Delta \cap E_j)$, with $E_j = G_j$ or $D_j$.
\medskip

Let us now consider the following two cases separately:
%
$(3.i)$ $\pi/3 \le \theta < \pi/2$, and $(3.ii)$ $\pi/2 \le \theta \le \pi$. (Recall that $\theta$ is the interior angle at the vertex $v$ of the triangle $uvw$.)

\medskip

{\sl The case $(3.i)$: $\pi/3 \le \theta < \pi/2$}. We claim that
\[ \int_{\{\beta \ge |\alpha|/2 \}} \beta \,a \xi \ge c \,G
\]
where we put
\[ G = \int_0^1 \int_0^1 \int_0^1 |s_2-s_1| \cdot H_j \cdot x^j \, \xi \, dt_3 dt_2 dt_1 .
\]
Recall that $H_j = \prod_{k=j+1}^N |z_k|$ and $z-b = x + i y$.

This may be seen as follows.
Fix $t_1 \in [0,1]$. Let $t_2 (t_1)$ be the smallest value of $t_2 \in [0,1]$ such that $\beta \ge \alpha/2 >0$, i.e. $\Im (s_2 (t_2) - s_1 (t_1) ) \ge (1/2)\Re (s_2 (t_2) - s_1 (t_1) ) >0$ for $t_2 \ge t_2(t_1)$.
If $|w-v|$ is much larger than $|v-u|$, then the term $x^j = [\Re(z-b)]^j$, which is comparable to $\Re[(z-b)^j]$ for $z-b\in \Delta$, may vary a lot in the triangle $uvw$. Thus, we split the integral into two parts. (This splitting is not
necessary when $|w-v| \le 2 |v-u|$, say.)
%

By our assumptions it follows that $1-t_2(t_1) \ge t_2(t_1)$ for $t_1 \in [0,1]$.
Note that the triangle with vertices at $u$, $v$ and $s_2 (2 \,t_2(0))$ is contained in the ball $B(v, 2 \rho \eps)$, centered at $v$, where $\rho =|v-b|$. Also, for all $z \in B(v, 2 \rho \eps)$, we have $x\approx \rho$. Thus, for $t_1\in [0, 1]$, we have
\begin{align}\label{t1fixed}
\int_{[t_2 (t_1),\, 1]} \int_0^1 |s_2 - s_1| \, H_j \cdot x^j \, \xi \, dt_3 dt_2 \ge  \int_{[2 t_2 (t_1),\, 1]} \int_0^1 |s_2 - s_1| \, H_j \cdot x^j \, \xi \, dt_3 dt_2 +
\end{align}
\[ + \, c \int_{[t_2 (t_1),\, 2t_2 (t_1)]} \int_0^1 |s_2 - s_1| \, H_j \cdot \rho^j \, \xi \, dt_3 dt_2 =: J_1 + c J_2 .
\]
%
%

Given $s_1=s_1(t_1)$, let $L = L(t_1)$ be the distance from $s_1$ to the segment $vw$. Then the lengths of segments $[s_1, s_2]$ with $s_2 = s_2(t_2)$ for any $t_2 \in [0,\, 2t_2 (t_1)]$ are all comparable to $L$. In fact, $L \le |s_1-s_2| \le 2L$.
Also, we have $\xi \approx |g(z)| \approx 1$ on $B$, where the implied constants depend only on $N$. These facts imply that
\[ J_2 \approx \int_{[0, \,t_2 (t_1)]} \int_0^1 |s_2 - s_1| \,H_j \cdot x^j \,\xi \, dt_3 dt_2 =: J_3
.\]
Thus, integrating both sides of the inequality \eqref{t1fixed} in $t_1 \in [0,1]$ gives
\[ \int_{\{\beta \ge \alpha/2 > 0 \}} \beta a \xi \gc \int_0^1 J_1 + c\int_0^1 J_2 \ge \int_0^1 J_1 + {c\over 2} \int_0^1 J_2 + {c\over 2} \int_0^1 c_1 J_3 \gc G
\]
since $G \approx \int_0^1 (J_1 + J_2 + J_3) \,dt_1$.
\\

%

Hence, it follows from \eqref{b0ineq} that
\[  \int_0^1 \int_0^1 \int_0^1 \Im[(s_2-s_1)\phi'''(z)] \, dt_3 \, dt_2 \, dt_1 =
\int \beta a \xi +  \int\alpha a \eta +  O(\eps G)
\]
\[  \ge \int_{\{\beta \ge \alpha/2 > 0 \}} \beta a \xi - b_0^{-1} \int |\alpha| a\xi + O(\eps G)
 \ge c_2 \, G - b_0^{-1} C_3 G + O(\eps G)
 \]
\[ \ge (c_2 - b_0^{-1} C_3 - C_4 \eps) G \ge {c_2\over 2} G
\]
if $b_0$ is chosen sufficiently large and $\eps$ sufficiently small.
Therefore, we may conclude that
\begin{align}\label{jaclowerbd}
|J (u, v, w)| &\gc
|v-u||w-v| \Big| \int_0^1 \int_0^1 \int_0^1 |s_2-s_1| \cdot H_j \, x^j \, dt_3 dt_2 dt_1 \Big| \\ \notag
&\gc |v-u||w-v| \cdot \Big| \int_0^1 \int_0^1 \int_0^1 (s_2-s_1) \cdot H_j \cdot (z-b)^j \, dt_3 dt_2 dt_1 \Big| \\ \notag
&= \Big| \int_u^v \int_v^w \int_{s_1}^{s_2} H_j \cdot  (z-b)^j \, dz ds_2 ds_1 \Big| . \notag
\end{align}
Here we used the fact that $\xi \approx 1$.
\medskip


Next, observe that the last integral is precisely the determinant of the Jacobian of the transformation $(u_1, u_2, u_3) \mapsto \sum_{j=1}^3 \Gamma(u_j)$ if we take $\Gamma(z) = (z, z^2/2, \psi(z))$ with $\psi'''(z) = H_j \, (z-b)^j$. Therefore, one can use Lemma \ref{jacest} to show that the last integral is bounded below by a constant multiple of
\begin{align*}
H_j \, \cdot V &(u,v,w) \cdot \Big| \sum_{a_1+a_2+a_3=j} (u-b)^{a_1} (v-b)^{a_2} (w-b)^{a_3}\Big|
\\
& = H_j \, \cdot V(u,v,w)\cdot |P_j (u-b, v-b, w-b)|
\\
& \gc V(u,v,w) \max_{i=1,2,3} [H_j \, |u_i -b|^j] \\
&\approx V(u_1, u_2, u_3) \max_{i=1,2,3} |\phi'''(u_i)|
\end{align*}
for $u_1$, $u_2$, $u_3 \in b+ B$, with $B \subset S_1 \cap\Delta \cap E_j$, where $E_j = G_j$ or $D_j$. Here, $P_j$ is as in Definition \ref{PN}, and we wrote $u_1=u$, $u_2=v$, and $u_3 =w$. (To get the inequality above, we argue as in Lemma \ref{jacest}, using the fact that $0< y < \eps x$ in $\Delta$.) This yields the desired lower bound \eqref{Jmax} when $\pi/3 \le \theta < \pi/2$.
(Recall that $\theta$ is the interior angle at the vertex $v$ of the triangle $uvw$.)

\medskip


{\sl The case $(3.ii)$: $\pi/2 \le \theta \le \pi$}.
In this case, we have $\alpha \ge 0$ and $\beta \ge 0$ (or $\beta \le 0$).
This case is easier than the previous one, since there is no cancelation in either of the integrals $\int\alpha a \xi$ or $\int \beta a \xi$. Hence, in this case we have $\int \alpha a \xi + |\int \beta a \xi| = \int \alpha a \xi + \int |\beta| a \xi \ge c\, G$. If $\int |\beta| a \xi \ge (c/2) G$, then we get $|\int\Im[(s_2-s_1)\phi'''(z)]| \gc G$, as before. If not, then we have $\int\alpha a \xi \ge (c/2) G$, and so we would get $|\int\Re [(s_2-s_1)\phi'''(z)]| \gc G$, instead. In either case, we obtain \eqref{jaclowerbd} for $u_i \in b+ B \subset b+ (S_1 \cap\Delta \cap E_j)$, $1\le i\le 3$, and the rest of the argument is the same as the previous case $(3.i)$.
\medskip

\noindent
{\sl Step 4. Completion of the proof of Lemma \ref{J-simplepoly}.}

We will finish the argument by stating how to make up the collection $\{ B \}$ to cover $\bbC$. The sets $\{ B \}$, which arose from all the decomposition steps above, need to be translated by $b$, and then one gets $b+S_1 = \cup (b+ B)$, except for a null-set.
To be precise, each distinct root $b$ of $\phi'''(z)$ contributes its own collection $\{ b+ B \}$ to cover $b+ S_1$, where $S_1 = S_1(b)$ depends on $b$. In fact, $b+S_1(b) = \{ z\in \bbC: ~ |z-b| < |z- b'|, ~ \forall b' \not= b\}$, where $\{ b'\}$ is the zero set of $\phi'''(z)$. Finally, the collection of all these sets gives the desired decomposition of $\bbC$, i.e. $\bbC = \cup_b [b+S_1(b)] = \cup_b \cup_{B\subset S_1(b)} (b+ B)$, ignoring a null-set. It just remains to rename the sets $b+B$ as $B$ so that $\bbC = \cup B$, except for a null-set. This completes the proof of Lemma \ref{J-simplepoly}.
\end{proof}

{\sl A sublevel set estimate.}
We also need the following simple observation on the complex form of the Vandermonde determinant:
\begin{lemma}\label{VdMbound}
Let $v(h)=v(h_2, \cdots, h_d) = |V(z_1, \cdots, z_d)|$ denote the absolute value of the
complex Vandermonde determinant, where we put $h_j = z_j-z_1$,
$2\le j\le d$. Then
\[ |\{ h = (h_2, \cdots, h_d) \in \bbC^{d-1}: \ v(h) \le u\}| \le C_d \,
u^{4/d}, \quad \forall u>0 .
\]
\end{lemma}
\begin{proof}
Write $x=(x_2, \cdots, x_d)$ and $y=(y_2, \cdots, y_d)$, where
$x_j = \Re h_j$ and $y_j = \Im h_j$. Then the set $G = \{ h \in
\bbC^{d-1}: \ |v(h)| \le u\}$ is contained in $\{ x\in \bbR^{d-1}:
\ |v(x)| \le u\} \times \{ y\in \bbR^{d-1}: \ |v(y)| \le u\}$,
because $|v(h)| \ge |v(x)|$ and $|v(h)| \ge |v(y)|$. Thus it
follows from the corresponding result in the real case ({\it cf.}
\cite{DM1}, \cite{BOS1}) that the measure $|G|$ in $\bbR^{2(d-1)}$
is bounded by $C_d (u^{2/d})^2$.
\end{proof}
%


\section{Interpolation of multilinear operators with symmetries}
\label{interpolation}

The following lemma was proved in \cite{BOS3}. It is a variant of an
interpolation theorem for $r$-convex spaces obtained in
\cite{BOS1}. The original version for Banach spaces, sometimes called the `multilinear trick', goes back to Christ \cite{Ch}.

\begin{theorem}\label{multtrick}
Let $n\ge 3$ be an integer and let $0<r\le 1$. Suppose that $\delta_1,\dots, \delta_n$ are real numbers so that the
$\delta_i$ are not all equal for $i\ge 2$. Let $V$ be an $r$-convex \footnote{This means that there is a constant $C$ such that $\| \sum_{j=1}^M f_j\|_V^r \le C \sum_{j=1}^M \| f_j\|_V^r$ for all $M\ge 1$ and $f_j \in V$. It is crucial that $C$ is independent of $M$. The Lorentz space $L^{r,\infty}$ is known to be $r$-convex for $0<r<1$. ({\it cf.} \cite{kalton}, \cite{stw})} Lorentz space, and let
$\overline X=(X_0, X_1)$ be a couple of compatible complete quasi-normed spaces.
Let $T$ be a multilinear operator defined on $n$-tuples of $(X_0+X_1)$-valued sequences
and suppose that for every permutation $\pi$ on $n$ letters we have the
inequality
\begin{equation}\label{Thypothesis}
\|T(f_{\pi(1)},\dots, f_{\pi(n)})\|_{V} \le
 \|f_1\|_{\ell_{\delta_1}^r(X_1)}
\prod_{i=2}^n \|f_i\|_{\ell_{\delta_i}^r(X_0)}\,.
\end{equation}
Then there is a constant $C$ such that
\begin{equation}\label{Tconclusion}
\|T(f_1,\dots, f_n)\|_{V} \le C
\prod_{i=1}^n \|f_i\|_{\ell_{\sigma}^{n r} \big(\overline X_{{1\over n}, n r} \big)}\,,\quad
\sigma = {1\over n}\sum_{i=1}^n \delta_i .
\end{equation}
\end{theorem}

\section{Proof of Theorem \ref{comppolthm}}


{\sl About this section.} We will assume that the conclusion \eqref{Jmax} of Lemma \ref{J-simplepoly} is valid for a given $d\ge 3$, and then formally deduce from this assumption
the $d$-dimensional version of \eqref{affrestr-d3}, which is in the same form as \eqref{affrestr-d}. Actually, we will prove the dual estimate \eqref{T-Q}. Since Lemma \ref{J-simplepoly} has been established for $d=3$, this shows Theorem \ref{comppolthm}. We decided to present the proof in this way, showing most steps in general dimension $d\ge 3$, since they are needed again in the next section to prove Theorem \ref{compthm} for all $d\ge 3$.

\medskip

{\sl Offspring curves.}
Write $\gamma(z)= (z, z^2, \cdots, z^{d-1}, \phi(z))$, where $\phi(z) = \sum_{i=0}^N \alpha_i \, z^i$, $\alpha_i \in \bbC$.
Let us put
\begin{equation}\label{offsp}
\Gamma(z) = (P_1(z), \cdots, \, P_{d-1}(z), \, \phi (z))
\end{equation}
where $P_j(z)= z^j+$ {\it lower order terms}, and $\phi (z) = \sum_{i=0}^N \alpha_i \,z^i$ with some new coefficients $\alpha_i \in \bbC$.

The expression $\Gamma (z)$ is an analogue of the `offspring curves' in the terminology of \cite{D1} and \cite{DM1}. For instance, if $\Gamma(z)$ is as above with $|\alpha_i| \le 1$ and $|h_j| \le 2$, for $1\le j\le d$, then the expression $\Gamma_1 (z,h) = d^{-1} \sum_{j=1}^d \Gamma(z+h_j)$ is again in the form \eqref{offsp}, and the coefficients $\widetilde{\alpha}_i$ of the last component $\phi_1(z)$ of $\Gamma_1 (z,h)$ satisfy $|\widetilde{\alpha}_i| \le C(d,N)$ for some constant $C(d,N)$.
(See \eqref{Gzh} below.)

\medskip





{\sl Two crucial lower bounds.}
As in \cite{BOS3} (see Section 4 there), the following two lower bounds will play crucial roles here. The first concerns the (real) {\it Jacobian} of the transformation $(z_1, \cdots, z_d) \mapsto \sum_{j=1}^d \Gamma(z_j)$, considered as a real mapping, and the second is about the {\it torsion} $\tau(z,h)$ of the offspring curves given by $z\mapsto \Gamma(z,h) = \sum_{j=1}^d \Gamma(z+ h_j)$.

\medskip

\noindent
(i) {\sl The Jacobian bound:}
\begin{align}\label{Jacbound}
J_\bbR (z_1, \cdots , z_d) \ge c(d,N) \, V(z_1, \cdots , z_d)^2 \max_{j=1,\cdots ,d} w(z_j)^\frac{d^2+d}{2} .
\end{align}
%

\noindent
(ii) {\sl The torsion bound:}
\begin{equation}\label{torbound}
 |w(z, h_2,\cdots, h_d)| = |\tau(z,h)|^{4/(d^2+d)} \ge c(d,N) \max_{j=1,\cdots,d} w (z+ h_j)
\end{equation}
for $z_j = z+h_j \in B$, whenever $B$ is one of the sets in Lemma \ref{J-simplepoly}.
Here, $h=(h_2, \cdots, \, h_d)$ with $h_j \in \bbC$, $h_1=0$, and $w(z)$ is given by \eqref{compwt} with $\gamma(z)$ replaced by $\Gamma(z)$.

These are \eqref{Jacobian-a}
and \eqref{tauzh} below, respectively. The precise statements can be found there.
We emphasize that for our argument to work (more precisely, for the use of Lemma \ref{multtrick} to be valid), at least one of these two lower bounds must be in the stronger form where, on the right-hand side of the inequality, instead of the usual {\it geometric mean} the {\it arithmetic mean} (or equivalently, the {\it maximum} as written above) of the relevant terms is used.

\medskip
The following proof is an adaptation of an argument used already in \cite{BOS3}. It is arranged somewhat differently here, because unlike in \cite{BOS3} we cannot assume that the result is known for the `nondegenerate' case in this context. Thus, both the nondegenerate and degenerate cases are treated simultaneously here. We give the proof in some detail, for some of the necessary changes may not be obvious. But our presentation will be somewhat sketchy at places. We refer the reader to sections 4 and 5 of \cite{BOS3} for more details on such points.

Observe that it suffices to consider the case $N\ge d$, since for $0\le N < d$, we have $\gamma^{(d)}(z) \equiv 0$, and so $w(z) \equiv 0$ and there is nothing to prove.
By a scaling argument it suffices to prove the estimate for functions $f$
supported in a fixed ball, say, $B(0,1)$ in $\bbC$ or $\bbR^2$.


Define
\[ w (z) = |\det (\Gamma '(z), \,\Gamma''(z), \cdots,\, \Gamma^{(d)}(z))|^{4/(d^2+d)} .\]
A calculation shows that
\begin{equation}\label{wbpower}
w(z)^{\frac{d^2+d}{4}} = c_d \, |\phi^{(d)}(z)|
\end{equation}
where $c_d = 2! \cdots (d-1)!$.

Now, for $\la >1$, define
\begin{equation}\label{Tla}
T_\la^{\Gamma} f(x)=  \psi(x) \int_{B(1)} e^{i\la x\cdot
\Gamma (z)} f(z) \, w (z) d\mu(z), \quad x\in \bbR^{2d} ,
\end{equation}
where $\psi(x)$ is a nonnegative cutoff function and $B(r) = B(0,r)$, $r>0$.

Put $Q =q_d = (d^2+d+2)/2$, and define
\begin{equation}\label{Al}
A_\la = \la^{2d/Q} \cdot \sup_{\Gamma}\| T_{\la}^{\Gamma}
\|_{L^{Q}(B(1) , \, w d\mu )\rightarrow L^{Q,\infty}(\bbR^{2d})}
\end{equation}
where the supremum is taken over all offspring curves $\Gamma$ as in \eqref{offsp} with $|\alpha_i|\le 1$.
%
(Notice that the cutoff function $\psi(x)$ in \eqref{Tla} may be replaced by a translation $\psi(x - x_0)$ without affecting the norm bound, since a factor of the form $e^{i\la x_0 \cdot \Gamma (z)}$ may be absorbed into the function $f(z)$.)

\medskip

Let us first see that $A_\la < \infty$ for each $\la >1$. By
H\"older's inequality and \eqref{wbpower} we have
\begin{align*}
 \| w \|_{L^1(B(1) , \, d\mu )} &\le |B(1)|^{\frac{d^2+d-4}{d^2+d}}\cdot \| w^{\frac{d^2+d}{4}} \|_{L^1(B(1) , \, d\mu
)}^{\frac{4}{d^2+d}} \\
&\le | B(1) |^{\frac{d^2+d-4}{d^2+d}} \cdot (c_d)^{\frac{4}{d^2+d}} \| \phi^{(d)}\|_{L^1(B(1) , \, d\mu
)}^{\frac{4}{d^2+d}} \le C_{d,N}
\end{align*}
for some constant $C_{d,N}$ uniform in the coefficients $\alpha = (\alpha_0, \cdots, \alpha_N)$ with
$|\alpha_i|\le 1$, $0\le i\le N$.

So, by H\"older's inequality we
obtain
$$ \| f\|_{L^1 (B(1), \, w d\mu )} \le
\| w \|_{L^1(B(1) , \, d\mu )}^{1/Q'} \| f\|_{L^{Q}(B(1),
\, w d\mu )} \le C_{d,N}^{1/Q'} \| f\|_{L^{Q}(B(1), \, w d\mu
)} .$$
Since $| T_{\la}^{\Gamma} f (x)| \le |\psi(x)| \cdot \| f\|_{L^1 (B(1) , \, w d\mu )}
$, the last inequality implies that
\begin{align*}
\| T_{\la}^{\Gamma} f \|_{L^{Q,\infty}(\bbR^{2d})} &\le \|
\psi \|_{L^{Q,\infty}(\bbR^{2d})} \| f\|_{L^1 (B(1) , \, w d\mu )}\\
& \le \|\psi \|_{L^{Q,\infty}(\bbR^{2d})} \cdot C_{d,N}^{1/Q'} \| f\|_{L^{Q}(B(1) , \, w d\mu
)} .
\end{align*}
Hence, it follows that for each $\la >1$,
\begin{equation}\label{Alafinite} A_\la \le \la^{2d/Q}\cdot
C_{d,N}^{1/Q'} \|\psi\|_{L^{Q,\infty}(\bbR^{2d})}  < \infty .
\end{equation}

Our goal is to show that $A_\la \le C(d,N)$, independent of $\la
>1$. This, in turn, would imply that
\begin{equation}\label{Tla-est}
\| T_{\la}^{\Gamma} f \|_{L^{Q,\infty}(\bbR^{2d})} \le C(d,N) \,
\la^{-2d/Q} \| f\|_{L^{Q}(B(1), \, w d\mu )}, \quad \la >1
\end{equation}
uniformly in $\alpha = (\alpha_0, \cdots, \alpha_N)$ with
$|\alpha_i|\le 1$, $0\le i\le N$, if $f$ is supported in $B(1)$.
Assuming \eqref{Tla-est}, it is easy to finish the proof of
\eqref{T-Q}. First we take $\Gamma (z) = \gamma(z)$.
Then we make a change of variables $x\mapsto \la^{-1} x$ to remove
the factor $\la^{-2d/Q}$, and next we take the limit as $\la \rightarrow
\infty$ to remove the cutoff function $\psi(x)$. Finally, summing over the $B$'s, we obtain
\eqref{T-Q} for $f$ supported in $B(1)$. Then a scaling argument extends \eqref{T-Q} to functions $f$ supported in $\bbC$.

\medskip

It remains to show $A_\la \le C(d, N)$ and \eqref{Tla-est}. Fix $\la >1$. Also fix $\Gamma(z)$ as in \eqref{offsp} with $\alpha =(\alpha _0, \cdots, \alpha _N)$, $\alpha _i \in \bbC$ and $|\alpha _i|\le 1$, $0\le i\le N$. Let $|h_j| \le 2$, $1\le j\le d$.
Put $\Gamma(z,h) = \sum_{j =1}^d \Gamma (z+h_j)$. Then
$\Gamma(z,h)$ is in the form
\begin{equation}\label{Gzh}
\Gamma(z,h) = ( d \cdot P_1(z), \cdots, \, d \cdot P_{d-1}(z),\, \phi_1(z))
\end{equation}
where the $P_j(z)$ are as in \eqref{offsp}, with the leading coefficient 1, but some new coefficients for the lower order terms, and $\phi_1(z) = \sum_{i=0}^N \tilde{a}_i \, z^i$ with $|\tilde{\alpha}_i | \le c_* = C(d,N)$. The constants $d, \cdots, \, d, c_*$ and $c_*$ can be factored from $\Gamma (z,h)$ and incorporated into $x$. Namely, we may rewrite
 \[ x\cdot  \Gamma(z,h) = y \cdot \Gamma_1 (z,h) .
 \]
 Here, $y= (d \,x_1, \cdots, d\, x_{d-2}, c_* \,x_{d-1}, c_* \,x_{d})= x L$, where $L$ is a $d\times d$ diagonal matrix, and
 \[ \Gamma_1 (z,h) = ( P_1(z), \cdots, \,P_{d-1}(z), \, c_*^{-1} \, \phi_2(z))
  \]
is an offspring curve as in \eqref{offsp}, of which the last component has coefficients $\widetilde{\alpha_i}$ with $|\widetilde{\alpha_i}| \le 1$.
The change of variables $x\mapsto y$ changes the cutoff function to
\[ \psi(y \,L^{-1}) = \psi\Big( {y_1\over d}, \cdots, \,{y_{d-2}\over d}, \,
{y_{d-1}\over c_*}, \,{y_{d}\over c_*}\Big) .
\]
Since $\psi(y \,L^{-1})$ is bounded by the sum of no more than $C(d,N)$ translates of $\psi(y)$, we may apply the definition of
$A_\la$. This only increases the constant by a factor $C(d,N)$.

\medskip



By writing $B(1) = B(0,1)$ as a union of the sets $B(1) \cap B$, where the $B$ are as in
Lemma \ref{J-simplepoly},
we may assume that $f$ is supported in $B$. We may also assume that $B\subset B(1)$. (Otherwise, replace $B$ with $B(1) \cap B$.)
Thus, we may rewrite
\begin{equation}\label{Tla-n}
T_\la^{\Gamma} f(x)=  \psi(x) \int_{B} e^{i\la x\cdot
\Gamma (z)} f(z) \, w (z) \, d\mu(z), \quad x\in \bbR^{2d} .
\end{equation}
Let us put
\begin{align*}\label{}
&M_{\la} (f_1, \cdots, f_d)(x) = \prod_{j=1}^d (T_{\la}^{\Gamma}
f_j)(x) = \\
&= \psi(x)^d \int_{B^d} e^{i\la x\cdot \sum_{i=1}^d
\Gamma(z_j)} \prod_{j=1}^d [(f_j \, w)(z_j) ]\ d\mu(z_1)\cdots
d\mu(z_d) \\
&= \psi(x)^d \int_{B(2)^{d-1}} \int_{B_h} e^{i\la x\cdot
\Gamma(z,h)} \prod_{j=1}^d [(f_j \, w)(z+ h_j)]\ d\mu(z)\,
d\mu(h_2)\cdots d\mu(h_d) .
\end{align*}
Here, $B_h$ is the intersection of the sets $B - h_j$ (translates of $B$)
over the indices $j=1, \cdots, d$.

Next, as in \cite{BL} we define the decomposed operators
\begin{equation*}\label{}
M_{\la, k} (f_1, \cdots, f_d)(x) =
\end{equation*}
\begin{equation}\label{multint}
\psi(x)^d \int_{S_k} \int_{B_h} e^{i\la x\cdot
\Gamma(z,h)} \prod_{j=1}^d [(f_j \, w)(z+ h_j)]\ d\mu(z)\,
d\mu(h_2)\cdots d\mu(h_d)
\end{equation}
where $S_k = {\{h\in B(2)^{d-1}:\ 2^{-k-1} < v(h) \le 2^{-k} \}}$, for $k\in \bbZ$, and
$v(h) = |V(z+ h_1, \cdots, z+ h_d)| = \prod_{1\le i<j \le d} |h_j - h_i| $ is
the Vandermonde determinant.

\medskip



{\sl An estimate at $q=Q$.}
By the considerations about $\Gamma(z,h)$ given in the paragraph containing \eqref{Gzh} and from the definition \eqref{Al} of $A_\la$, it follows that
\begin{equation}\label{osest}
\Big\| \psi(x) \int_{B_h} e^{i\la x\cdot \Gamma(z,h)}
f(z) \cdot w(z,h) \ d\mu(z) \Big\|_{Q,\infty} \lc \, \la^{-
{2d\over Q}} A_\la \| f \|_{L^Q (B_h,\, w(z,h) d\mu )}
\end{equation}
uniformly in $h$. Here, $w(z,h) = |\tau(z,h)|^{4/(d^2+d)}$, and
\[ \tau(z,h) = \det(\Gamma'(z,h), \cdots ,\Gamma^{(d)}(z,h)) .\]
We have
\begin{equation*}\label{}
|\tau(z,h)|
= d^{d-1} 2!\cdots(d-1)! \cdot \Big| \sum_{j=1}^{d} \phi^{(d)}(z+h_j)\Big| .
\end{equation*}
Thus, as in the proof of Lemma \ref{J-simplepoly} (or Lemma \ref{jacest}) we obtain
\begin{equation}\label{tauzh}
 |\tau(z,h)| \ge c_{d,N} \sum_{i=1}^d \big|\phi^{(d)}(z+h_i)\big| \ge c_{d,N} \max_{1\leq i \leq d} w (z+ h_i)^{(d^2+d)/4}
\end{equation}
for $z+h_i \in B$. Here we used \eqref{wbpower}. Now set
\begin{equation*}\label{}
w_* (z,h) := \prod_{i=1}^d w(z+ h_i)^{a_i}
\end{equation*}
for some constants $a_i \in [0,1]$ with $\sum_{i=1}^d a_i =1$. We
will choose $a_i$ suitably later so that
the condition $\delta_2
\not= \delta_3$ (the $\delta_j$ are to be defined later) is satisfied, which will allow us to apply the interpolation theorem, Theorem \ref{multtrick}. (See the paragraph containing \eqref{ellq} below.)
Thus, as was mentioned in the first paragraph of this section, the fact that we have an arithmetic mean instead of a geometric mean as the lower bound in \eqref{tauzh} plays a key role in our argument.

The inequality \eqref{tauzh} implies that
\begin{equation}\label{ostor}
w(z,h) \ge c \, w_*(z,h)
\end{equation}
for $z \in B_h$, where $c = c(d,N) >0$ is a constant independent of $z$, $h$.
(Recall that $B_h$ is the intersection of the sets $- h_j + B$, $j=1, \cdots, d$.)

If we write $w_* (z,h) = g(z,h) w(z,h)$ for some nonnegative function $g(z,h) \le c$, then
we may apply \eqref{osest} with $f(z,h)$ replaced by $f(z,h) g(z,h)$. Since $g(z,h)^Q \le C g(z,h)$, this gives
\begin{equation*}\label{}
\Big\| \psi(x) \int_{B_h} e^{i\la x\cdot \Gamma(z,h)}
f(z) \cdot w_*(z, h) \ d\mu(z) \Big\|_{Q,\infty} \le C \la^{-{2d\over Q}}
A_\la \| f \|_{L^Q (w_*(z, h) d\mu )} .
\end{equation*}
(See Observation 5.1 in \cite{BOS3} for more on this argument.)

It follows then from an analogue of Minkowski's inequality, by using an equivalent `norm' on $L^{Q,\infty}$ for this purpose (see Section 4 of \cite{BOS3}), that
\begin{equation*}\label{}
\| M_{\la, k} (f_1, f_2,\cdots, f_d)\|_{Q,\infty} \le
\end{equation*}
\begin{align*}\label{}
\le C &\int_{S_k} \Big\| \psi(x) \int_{B_h} e^{i \la
x\cdot
\Gamma(t,h)} \times \\
&\times \prod_{j=1}^d [f_j(z + h_j) w(z + h_j)^{1- a_j}] \cdot
w_*(z,h) \ d\mu(z) \Big\|_{Q,\infty} \ d\mu(h_2) \cdots d\mu(h_d)
\end{align*}
\begin{equation*}\label{}
\le C \, \la^{-{2d\over Q}} A_\la \int_{S_k} \Big\| \prod_{j=1}^d [f_j(z +
h_j) w(z + h_j)^{1- a_j}] \Big\|_{L^Q (w_* (z,h) d\mu )} d\mu(h_2)
\cdots d\mu(h_d)
\end{equation*}
\begin{equation*}\label{}
= C \, \la^{-{2d\over Q}} A_\la \int_{S_k} \Big\| \prod_{j=1}^d \Big[ f_j(z+
h_j) w(z + h_j)^{1- {a_j\over Q'}} \Big] \Big\|_{L^Q (d\mu )} d\mu(h_2)
\cdots d\mu(h_d) .
\end{equation*}

We will now apply H\"older's inequality to bound the inner norm and also use the sublevel set estimate in Lemma \ref{VdMbound} with $u=2^{-k}$.
This gives
\begin{equation}\label{Eq}
\| M_{\la, k} (f_1, \cdots, f_d)\|_{Q,\infty} \le C\, \la^{-{2d\over Q}}
A_\la \cdot 2^{-{4k\over d}}  \prod_{j=1}^d \Big\| f_j w^{1- {a_j\over Q'}}
\Big\|_{L^{q_j}(d\mu )}
\end{equation}
where $\sum_{j=1}^d 1/q_j = 1/Q$ for some numbers $q_j$, $1\le q_j\le \infty$, to
be chosen later.

Let us now put
$$ \Omega_i = \{ z\in \bbC: ~ 2^{i-1} \le w (z) < 2^i \}, \quad i\in \bbZ . $$
The triangle inequality implies that
\[ \|f \, w^{\alpha} \|_{L^{p}(d\mu )} = \|\sum_{i\in
\bbZ} \chi_{\Omega_i} f \, w^{\alpha} \|_{L^{p}(d\mu )} \le C
\sum_{i\in \bbZ} 2^{i \alpha} \|f \chi_{\Omega_i} \|_{L^{p}(d\mu
)} , \text{  for  } \alpha \in \bbR .\]
Hence, it follows that
\begin{equation}\label{asymm}
\| M_{\la, k} (f_1, \cdots, f_d)\|_{Q,\infty} \le C \la^{-{2d\over Q}} A_\la
\cdot 2^{-{4k\over d}}  \prod_{j=1}^d \| f_j \|_{\ell^1_{\alpha_j}
(L^{q_j}(d\mu ))}
\end{equation}
where we put $\alpha_j = 1 - a_j/Q'$. Here the expression $\|
f\|_{\ell^p_{\alpha} (X)}$ stands for
$$ \| \{f \chi_{\Omega_i}\}\|_{\ell^p_{\alpha} (X)} =
\Big(\sum_{i\in \bbZ} \big[ 2^{\alpha i} \| f
\chi_{\Omega_i}\|_{X} \big]^p \Big)^{1/p}  $$
where $X$ is a Banach space (or a complete quasi-normed space) of functions
on $\bbR^2$. Thus, we identify $f$ with the sequence $\{f
\chi_{\Omega_i}\}_{i\in \bbZ}$.

\medskip

{\sl An $L^2$ estimate.}
Next, it follows from B\'ezout's theorem
that the transformation $(z,h_2, \cdots, h_d) \mapsto \Gamma(z, h_2,\cdots, h_d)$ has bounded generic multiplicity $\le N \cdot (d-1)!$.
By Proposition 1.4.10 on p. 51 in \cite{Kr}, the Jacobian of this
transformation as a real mapping is given by
$$ J_\bbR (z_1, \cdots, z_d)= |J_\bbC (z_1, \cdots, z_d)|^2
= |\det(\Gamma'(z+h_1), \cdots, \Gamma'(z+h_d))|^2 $$
for $z_j = z+h_j \in B$.
Here, $J_\bbC (z_1, \cdots, z_d)$ (or $J_\bbC (z, h_2,\cdots, h_d)$) denotes the determinant of the holomorphic Jacobian matrix for
the transformation $(z,h) = (z,h_2,\cdots, h_d) \mapsto \Gamma(z, h) = \sum_{j=1}^d \Gamma(z+h_j)$.

For instance, when $d=3$, we have
\begin{equation}\label{phi-rep}
 J_\bbC (z_1, z_2, z_3) = \int_{z_1}^{z_2} \int_{z_2}^{z_3} \int_{s_1}^{s_2} \phi'''(z) \, dz\, ds_2 \,ds_1 .
\end{equation}
(For higher dimensions there is a similar representation, defined recursively, which involves integrals of $\phi^{(d)}(z)$. See \cite{BOS2, DeW, DeFW}.)

Hence, by our assumption that Lemma \ref{J-simplepoly} holds for $d\ge 3$, it follows that
\begin{align}\label{Jacobian-a}
J_\bbR (z, h_2,\cdots, h_d) \gc v(h)^2 \cdot {1\over d}\sum_{j=1}^d w(z+ h_j)^{d^2+d \over 2} \ge v(h)^2 \prod_{j=1}^d w(z+ h_j)^{d^2+d \over 2d}
\end{align}
if $z+ h_j \in B$. (See also Remark \ref{rem-Gamma}.)
Here the implied constant $c = c_{d,N} >0$ depends only on $d$ and $N$.

Next, we change variables in the integral \eqref{multint} and use the Plancherel theorem. Then we reverse the change of variables and use
\eqref{Jacobian-a} and the sublevel set estimate in Lemma \ref{VdMbound} to
obtain
\begin{align*}\label{}
&\| M_{\la, k} (f_1, \cdots, f_d)\|_2 \le C\la^{-d} \times\\
&\times \left(\int_{S_k} \int \prod_{j=1}^d | (f_j\, w)(z+h_j) |^2 \, J_\bbR (z, h)^{-1} d\mu (z) d\mu (h_2)\cdots d\mu (h_d) \right)^{1/2} \\
&\qquad \le C\la^{-d} \left(\int_{S_k} \int \prod_{j=1}^d | (f_j \,w^a)(z+h_j) |^2 \,v(h)^{-2} d\mu (z) d\mu (h_2)\cdots d\mu (h_d) \right)^{1/2} \\
&\qquad \le C \la^{-d}
\,2^k 2^{-\frac{2k}{d}} \| f_1 \,w^a \|_{L^2 (d\mu )} \prod_{j=2}^d \| f_j\, w^a \|_{L^{\infty}(d\mu )}
\end{align*}
for $a=(3-d)/4$.

By permuting the variables and interpolating the resulting estimates one gets
\begin{equation*}\label{}
\| M_{\la, k} (f_1, \cdots, f_d)\|_2 \le C \la^{-d}
\,2^{\frac{k(d-2)}{d}} \prod_{j=1}^d \| f_j \,w^a \|_{L^{r_j}(d\mu )}
\end{equation*}
for some numbers $1\le r_j \le
\infty$, to be chosen later, such that $\sum_{j=1}^d r_j^{-1} = 2^{-1}$.
Using the triangle inequality on each norm again gives
\begin{equation}\label{E2}
\| M_{\la, k} (f_1, \cdots, f_d)\|_2 \le C \la^{-d}
\,2^{\frac{k(d-2)}{d}} \prod_{j=1}^d \| f_j \|_{\ell^1_a (L^{r_j}(d\mu ))} .
\end{equation}

{\sl Summation of the estimates.}
By estimating the distribution function of the sum of $M_{\la, k}(f_1, \cdots, f_d)(x)$ over $k$,
using \eqref{asymm} and \eqref{E2},
we obtain the estimate
\begin{align*}\label{}
\Big| \Big\{ &\Big|\sum_{k=-\infty}^\infty M_{\la, k} \Big| > 2\alpha  \Big\} \Big| \le \Big| \Big\{ \Big| \sum_{2^k > \beta} M_{\la, k} \Big| > \alpha \Big\} \Big| + \Big| \Big\{ \Big| \sum_{2^k \le \beta} M_{\la, k} \Big| > \alpha \Big\} \Big| \\
& \le \la^{-2d} \, \Big({CA_\la\over \alpha}\Big)^Q \beta^{-{4Q\over d}} \prod_{j=1}^d \| f_j \|_{\ell^1_{\alpha_j} (L^{q_j})}^Q + \la^{-2d} {C^2 \over \alpha^2} \beta^{\frac{2(d-2)}{d}} \prod_{j=1}^d \| f_j \|_{\ell^1_a (L^{r_j})}^2
\end{align*}
for $\beta >0$. Choosing the value
\[  \beta = \Big( \alpha^{2-Q} \, A_\la^Q \,
 \, \prod_{j=1}^d \Big[\| f_j \|_{\ell^1_{\alpha_j} (L^{q_j}(d\mu ))}^Q \| f_j \|_{\ell^1_a (L^{r_j}(d\mu ))}^{-2}\Big] \Big)^{\frac{d}{2(d-2+2Q)}}
\]
yields that
\[ \| M_\la (f_1, \cdots, f_d) \|_{Q/d, \infty} \le C \, \la^{-\frac{2 d^2}{Q}} \, A_\la^\frac{d-2}{d+2} \,
\prod_{j=1}^d \| f_j \|_{\ell^1_{\alpha_j} (L^{q_j}(d\mu ))}^\frac{d-2}{d+2}  \| f_j \|_{\ell^1_a (L^{r_j}(d\mu ))}^\frac{4}{d+2}.
\]
Here we used the fact that $d-2+2Q = d(d+2)$ and $Q = (d^2 + d +2) /2$.
%
%

By Lemma A.3 in \cite{BOS3}, this implies that
\begin{align*}\label{}
\| M_{\la} (f_1, \cdots, f_d)\|_{Q/d,\infty} \le & C \, \la^{-\frac{2 d^2}{Q}}\,A_\la^{\frac{d-2}{d+2}} \times \\
&\times \prod_{j=1}^d \| f_j\|_{\big(\ell^1_{\alpha_j} (L^{q_j}(d\mu
)), \, \ell^1_a (L^{r_j}(d\mu )) \big)_{\frac{4}{d+2},1}}.
\end{align*}
From Lemma A.4 in \cite{BOS3}, we have
\[ \big( \ell^1_{\alpha_j} (L^{q_j}(d\mu )),
\, \ell^1_a (L^{r_j}(d\mu )) \big)_{\frac{4}{d+2},1} = \ell^1_{\beta_j}
\big( L^{p_j,1}(d\mu ) \big)
\]
where
\begin{equation*}\label{p1}
{1\over p_j} = {d-2 \over d+2}\ {1\over q_j} + {4\over d+2}\ {1\over r_j}~ \text{ and }~
\beta_j = {d-2 \over d+2}\ \alpha_j + {4\over d+2}\ a .
\end{equation*}
Thus, we obtain
\begin{align}\label{first}
\Big\| \prod_{j=1}^d T_\la^{\Gamma} & f_j \Big\|_{Q/d,\infty} \le
C  \la^{-\frac{2 d^2}{Q}} \, A_\la^{\frac{d-2}{d+2}} \times \\
& \times \| f_1\|_{\ell^1_{\beta_1} (L^{p_1,1}(d\mu ))}
\| f_2 \|_{\ell^1_{\beta_2} (L^{p_2,1}(d\mu ))}
\prod_{j=3}^d  \| f_j \|_{\ell^1_{\beta_j} (L^{p_j,1}(d\mu ))} . \notag
\end{align}

On the other hand we can get an alternative estimate by
taking $q_j = d Q$ and $\alpha_j = 1-1/(d Q')$ for all $j$ in
\eqref{asymm}, and also taking all $r_j=2d$ in \eqref{E2}. Then taking all $f_j = f$ gives
\begin{equation}\label{second}
\big\| T_\la^{\Gamma} f \big\|_{Q,\infty} \le
C \la^{-\frac{2 d}{Q}}\,A_\la^{\frac{d-2}{d(d+2)}} \| f
\|_{\ell^1_{\delta_0} (L^{Q,1} (d\mu ))}
\end{equation}
where $\delta_0 = 1/Q$.

\medskip

{\sl Preparation for interpolation.} We will now consider the $n$-linear symmetric operator $\prod_{j=1}^n T_\la^{\Gamma} f_j$ with some $n > Q$. Then we need to estimate its $L^{r,\infty}$ quasi-norm with $r=Q/n <1$. This is to take advantage of the $r$-convexity of this space. (See Section \ref{interpolation} and the footnote 3 there.) For simplicity of notation, let us take $n=dQ$.
By applying a variant of H\"older's inequality ({\it cf.} (2.1) in
\cite{BOS1}), using \eqref{first} for the first $d$ factors and \eqref{second} for the rest, we
get
\begin{align*}\label{}
\Big\| \prod_{j=1}^{dQ} T_\la^{\Gamma} f_j \Big\|_{1/d ,\infty}
& \le C (dQ)^d \, \la^{-2 d^2} \,A_\la^{Q\frac{d-2}{d+2}}
\| f_1 \|_{\ell^1_{\beta_1} (L^{p_1,1})} \| f_2 \|_{\ell^1_{\beta_2} (L^{p_2,1})}\times \\
&\qquad \times  \prod_{j=3}^d  \| f_j \|_{\ell^1_{\beta_j} (L^{p_j,1})} \prod_{j=d+1}^{dQ}
\| f_j \|_{\ell^1_{\delta_0} (L^{Q,1})} .
\end{align*}

Now we may choose $q_1, \cdots, q_d$, and $r_1, \cdots ,r_d$
(hence also $p_1, \cdots, p_d$) such that $p_1 \not= p_2$, with
$p_2$ strictly between $p_3$ and $Q = (d^2 + d + 2)/2$, and also that $ p_3 = \cdots = p_d$ and
\begin{align}\label{p2} {1\over p_2} = {d-2 \over d Q -2} {1\over p_3} + {d(Q-1) \over dQ-2} {1\over Q} .
\end{align}
Note that we have then also
$$ {1\over d} \Big({1\over p_1} + {1\over p_2} + \cdots + {1 \over p_d}\Big) = {1\over Q} .$$
(In fact, we may choose $q_j$ and $r_j$ such that $1/p_3 = 1/Q-\varepsilon$
for some small $\varepsilon \not= 0$.
Also take $1/p_2 = 1/Q - (d-2)\varepsilon/(dQ-2)$ and $1/p_1 = 1/Q + (dQ-1)(d-2)\varepsilon/(dQ-2)$.
These choices satisfy the requirements listed above.)

Put $r=1/d$ and bound each quasi-norm above of the form
$\|\cdot\|_{\ell^1_{\rho} (L^{p,1})}$  by the quasi-norm
$\|\cdot\|_{\ell^r_{\rho} (L^{p,r})}$. With $f_1$, $f_2$ fixed, let us permute the remaining functions and take generalized geometric means of the resulting estimates to get
\begin{align*}\label{}
\Big\| \prod_{j=1}^{dQ} T_\la^{\Gamma} & f_j \Big\|_{1/d ,\infty}
\le C\la^{-2 d^2}
\,A_\la^{Q\frac{d-2}{d+2}} \times \\
&\times \| f_1 \|_{\ell^r_{\beta_1} (L^{p_1,r})} \| f_2
\|_{\ell^r_{\beta_2} (L^{p_2,r})} \prod_{j=3}^{dQ} \| f_j
\|_{\ell^r_{\beta_j} (L^{p_j,r})}^{d-2 \over dQ-2}\ \| f_j
\|_{\ell^r_{\delta_0} (L^{Q,r})}^{d(Q-1) \over dQ-2} .
\end{align*}
By \eqref{p2}, Lemma A.3 and A.4 in \cite{BOS3}, we obtain
\begin{align*}\label{}
\Big\| \prod_{j=1}^{dQ} T_\la^{\Gamma} f_j \Big\|_{1/d ,\infty}
\le &C\la^{-2 d^2}
\,A_\la^{Q\frac{d-2}{d+2}} \times \\
&\times \| f_1 \|_{\ell^r_{\delta_1} (L^{p_1,r})} \| f_2
\|_{\ell^r_{\delta_2} (L^{p_2,r})} \prod_{j=3}^{dQ} \| f_j
\|_{\ell^r_{\delta_j} (L^{p_2,r})}
\end{align*}
where $\delta_1 = \beta_1$, $\delta_2 = \beta_2$ and
\[ \delta_j = {d-2 \over dQ-2} \beta_j + {d(Q-1) \over dQ-2} \delta_0, ~ 3\le j\le d .
\]

We may choose $a_j \in [0,1]$ such that
$\sum_{j=1}^d a_j=1$, and
$\delta_2 \not= \delta_3$. (Recall that
$\beta_j = [(d-2)/(d+2)] \alpha_j + [4/(d+2)] a$, $ \alpha_j = 1 - a_j/Q'$
and $a = (3-d)/4$. Thus, it is easy to see that we can satisfy the condition $\delta_2 \not= \delta_3$, by choosing $a_2$ and $a_3$ suitably.)

\medskip

{\sl Application of the interpolation theorem.}
We are now in a position to apply Theorem \ref{multtrick}.
Let us take $X_0 = L^{p_2,r}(d\mu )\text{  and  } X_1 = L^{p_1,r}(d\mu )$.
It follows from \eqref{Tconclusion} with $n = d Q$ and $V = L^{r, \infty}$
for $r=1/d$ that
\begin{align*}\label{}
\Big\| \prod_{j=1}^{dQ} T_\la^{\Gamma} f_j \Big\|_{1/d ,\infty}
\le C \la^{-2 d^2} \, A_\la^{Q\frac{d-2}{d+2}} \prod_{j=1}^{dQ} \| f_j
\|_{\ell^{Q}_s (\overline{X}_{{1\over n}, Q})}
\end{align*}
where $ s= [\delta_1 + \delta_2 + (n -2)\delta_3 ]/n$. Taking
all $f_j=f$ yields
\begin{equation}\label{ellq}
\| T_\la^{\Gamma} f\|_{Q,\infty} \le C \la^{-2d/Q}\,A_\la^{\frac{d-2}{d(d+2)}}
\| f\|_{\ell^{Q}_s (\overline{X}_{{1\over n}, Q})} .
\end{equation}
Note that we have $s=1/Q =2/(d^2 + d +2)$, since
\begin{align*}
dQ s = \sum_{j=1}^{dQ} \delta_j &=\delta_1+\delta_2+ (dQ-2) \Big(
{d-2 \over dQ-2} \beta_3 +
{d(Q-1)\over dQ-2}{1\over Q}\Big)\\
&= {d-2 \over d+2}\sum_{j=1}^d \alpha_j + { d(3-d) \over d+2} + {d(Q-1)\over Q}\\
&= {d-2 \over d+2} \Big( d - {1\over Q'} \Big) + {d(3-d) \over d+2} + {d \over Q' } = d .
\end{align*}

Moreover, we have
\begin{align*}
\overline{X}_{{1\over n}, Q} = (X_0, X_1)_{{1\over n},Q} = (L^{p_2,r},
L^{p_1,r})_{{1\over n}, Q} = L^{p,Q} = L^Q ( d\mu )
\end{align*}
since $p_1 \not= p_2$ and
\begin{align*}
 {1\over p} &:= {1\over n} {1\over p_1} +  {n - 1\over
n} {1\over p_2} = {1\over dQ} \Big({1\over p_1} + {1\over p_2} + {dQ -2 \over
 p_2}\Big)
 = {1\over Q}
\end{align*}
by the choice of $p_1, \cdots, p_d$ made above in the paragraph containing \eqref{p2}. Here we also used the fact ({\it cf.} Theorem 5.3.1 in \cite{BeL}) that if $p_0\not=p_1$, then
$$(L^{p_0, r_0}, L^{p_1, r_1})_{\theta, s} = L^{p, s}$$
for $1/p=(1-\theta)/p_0 + \theta/p_1$, $0<\theta <1$, and $0< s\le
\infty$. (As usual, $p_j$, $r_j \in (0, \infty]$, and we assume
$r_j=\infty$ when $p_j=\infty$.)

This shows that we have
\begin{align*}
 \| f\|_{\ell^{Q}_{s} \big( \overline{X}_{{1\over n}, Q}\big)} &= \| \{f
\chi_{\Omega_k}\} \|_{\ell^Q_{1/Q} (L^Q (d\mu ))} \\
&= \Big(\sum_{k\in \bbZ} \big[ 2^{k /Q} \| f
\chi_{\Omega_k}\|_{L^Q(d\mu )} \big]^Q \Big)^{1/Q} \approx \| f\|_{L^Q (w d\mu )}
\end{align*}
where the last equivalence is a consequence of the fact that $w(z) \approx 2^k$ for $z\in \Omega_k$.
So, \eqref{ellq} implies that
\begin{equation*}\label{}
\| T_\la^{\Gamma} f\|_{Q,\infty} \le C_{d,N} \la^{-\frac{2 d}{Q}} A_\la^{\frac{d-2}{d(d+2)}}
\| f\|_{L^{Q}(w d\mu )}
\end{equation*}
with a constant $C_{d,N}$ independent of $\la >1$ and $\Gamma$ with $|\alpha_i| \le 1$.

Hence, by the definition \eqref{Al} of $A_\la$, we obtain
\begin{equation*}\label{}
\,A_\la \le C_{d,N} \, A_\la^{\frac{d-2}{d^2+2d}} .
\end{equation*}
Since we have $A_\la < \infty$ for $\la >1$ by \eqref{Alafinite}, it follows
that $A_\la \le C(d,N) = (C_{d,N})^{ (d^2 + 2d) / (d^2 + d +2)}$, for all
$\la >1$. Therefore, we may conclude that the estimate
\begin{equation*}\label{}
\| T_\la^{\Gamma} f\|_{Q,\infty} \le  C(d,N) \la^{-\frac{2 d}{Q}} \,
 \| f\|_{L^{Q}(w d\mu )}
\end{equation*}
holds for $Q= (d^2 + d +2) /2$, uniformly in $\la >1$ and $\Gamma$. This completes the proof of \eqref{Tla-est}. Finally, we take $C(N) = \sum_{d=1}^N C(d,N)$. Taking $d=3$ gives the dual estimate of \eqref{affrestr-d3}. \qed

\section{Proof of Theorem \ref{compthm}}


The proof in the previous section carries over here with minor modifications. Thus, we only need to
indicate how to modify the argument to work in this situation.
Here we define offspring curves by
\[ \Gamma_b(z) = {1\over m} \sum_{i=1}^m \gamma(z+ b_i)
\]
where $b_i \in \bbC$ and $b_1 =0$. Again, by a scaling argument it suffices to prove the estimate for functions $f$
supported in $B(0,1)$ in $\bbC$ or $\bbR^2$.
We only need to divide $B(0,1)$ into a bounded number
of narrow sectors centered at the origin. By rotation (which is possible by
the homogeneity of $\phi(z)=z^N$ as in Section 3), it is enough to show the
estimate for $f$ supported in $\Delta = \{z = x+ iy \in B(0,1) :~
0< y < \varepsilon x \}$ with some small $\varepsilon = \varepsilon(d, N) >0$.

%
%
%
%

Define
\begin{equation}\label{Tla-d}
T_\la^{\Gamma_b} f(x)=  \psi(x) \int_{\Delta_b} e^{i\la x\cdot
\Gamma_b (z)} f(z) \, w_b (z) d\mu(z), \quad x\in \bbR^{2d} ,
\end{equation}
where $\psi(x)$ is a nonnegative cutoff function and $\Delta_b =
\bigcap_{i=1}^m (\Delta - b_i) \subset \Delta$. (Here, $\Delta -
a = \{ z - a  : z \in \Delta \}$ denotes a translation of $\Delta$.)

Recall that $Q =q_d =(d^2 +d +2)/2$. Define
\begin{equation}\label{Al-d}
A_\la = \la^{2d/Q} \cdot \sup_{\Gamma_b}\| T_{\la}^{\Gamma_b}
\|_{L^{Q}(\Delta_b, \, w_b d\mu )\rightarrow L^{Q,\infty}(\bbR^{2d})}
\end{equation}
where the supremum is taken over all $\Gamma_b$, with $b= (b_1, \cdots, b_m) \in \bbC^{m}$, $m \in \mathbb{N}$, $b_1=0$, and $|b_j|\le 1$, for $1 \le i \le m$. (Note that $\Delta_b$ is empty, if $|b_j|>1$ for some $i$.)

Let us show that $A_\la < \infty$, for each $\la >1$. By
H\"older's inequality and \eqref{wbpower} we have
\begin{align*}
 \| w_b \|_{L^1(\Delta_b , \, d\mu )} &\le |\Delta_b|^{\frac{d^2+d-4}{d^2+d}}\cdot \| w_b^{\frac{d^2+d}{4}} \|_{L^1(\Delta_b , \, d\mu
)}^{\frac{4}{d^2+d}} \\
&\le | \Delta |^{\frac{d^2+d-4}{d^2+d}} \cdot \Big( m^{-1} \sum_{j=1}^m \| \phi^{(d)}(\cdot+b_j)\|_{L^1(\Delta-b_j , \, d\mu
)} \Big)^{\frac{4}{d^2+d}}\\
&\le | \Delta |^{\frac{d^2+d-4}{d^2+d}} \cdot \| \phi^{(d)}\|_{L^1(\Delta , \, d\mu
)}^{\frac{4}{d^2+d}} \le C_{d,N}
\end{align*}
for some constant $C_{d,N}$ independent of $m \ge 1$ and $b$. So, by H\"older's inequality we
obtain
$$ \| f\|_{L^1 (\Delta_b , \, w_b d\mu )} \le
\| w_b \|_{L^1(\Delta_b , \, d\mu )}^{1/Q'} \| f\|_{L^{Q}(\Delta_b,
\, w_b d\mu )} \le C_{d,N}^{1/Q'} \| f\|_{L^{Q}(\Delta_b, \, w_b d\mu
)} .$$
Since $| T_{\la}^{\Gamma_b} f (x)| \le |\psi(x)| \cdot \| f\|_{L^1 (\Delta_b , \, w_b d\mu )}
$, the last inequality implies that
\begin{align*}
\| T_{\la}^{\Gamma_b} f \|_{L^{Q,\infty}(\bbR^{2d})} &\le \|
\psi \|_{L^{Q,\infty}(\bbR^{2d})} \| f\|_{L^1 (\Delta_b , \, w_b d\mu )}\\
& \le \|\psi \|_{L^{Q,\infty}(\bbR^{2d})} \cdot C_{d,N}^{1/Q'} \| f\|_{L^{Q}(\Delta_b, \, w_b d\mu
)} .
\end{align*}
Hence, it follows that for each $\la >1$,
\begin{equation}\label{Alafinite-d} A_\la \le \la^{2d/Q}\cdot
C_{d,N}^{1/Q'} \|\psi\|_{L^{Q,\infty}(\bbR^{2d})}  < \infty .
\end{equation}

%
%

It remains to show $A_\la \le C(d,N)$, uniformly in $\la >1$. Fix $\la >1$ and
$b$ such that $|b_i|\le 1$, $1\le i\le m$, and put $\Gamma(z,h) = \Gamma_b (z,h) = \sum_{j =1}^d \Gamma_b (z+h_j
) = m^{-1}\sum_{j =1}^d \sum_{i=1}^m \gamma(z+ b_i + h_j )$, with
$h=(h_2, h_3,\cdots,h_d )$, $h_1 = 0$ and $z+ b_i + h_j \in \Delta$.

Now set
\begin{equation*}\label{}
M_{\la} (f_1, f_2, \cdots, f_d)(x) = \prod_{j=1}^d (T_{\la}^{\Gamma_b}
f_j)(x) =
\end{equation*}
\begin{equation*}\label{}
\psi(x)^d \int \int_{\Delta_{b,h}} e^{i\la x\cdot
\Gamma(z,h)} \prod_{j=1}^d [f_j(z+ h_j) w_b(z+ h_j)]\ d\mu(z)\,
d\mu(h_2) \cdots d\mu(h_d) .
\end{equation*}
where $\Delta_{b,h} = \bigcap_{j=1}^d \bigcap_{i=1}^m(\Delta - b_i - h_j)$.

As before, define the decomposed operators by
\begin{equation*}\label{}
M_{\la, k} (f_1, f_2, \cdots ,f_d)(x) =
\end{equation*}
\begin{equation*}\label{}
\psi(x)^d \int_{S_k} \int_{\Delta_{b,h}} e^{i\la x\cdot
\Gamma(z,h)} \prod_{j=1}^d [f_j(z+ h_j) w_b (z+ h_j)]\ d\mu(z) \,
d\mu(h_2) \cdots d\mu(h_d)
\end{equation*}
where $S_k = {\{h\in B(1)^{d-1}:\ 2^{-k-1} < v(h) \le 2^{-k} \}}$, $k\in \bbZ$.

Note that $\Gamma(z,h)$ may be written in the form $d \cdot
(dm)^{-1} \sum_{i=1}^{dm} \gamma(z+ c_i)$ for some $c_i$. (In
fact, we may take $c_i = b_j+h_k$ with $c_1 = b_1 + h_1 =0$ and
the rest numbered in some way.) Thus, $\Gamma(z,h)$ is an
offspring curve except for the factor $d$. To remove the $d$, we make
the substitution $y=d \cdot x$, which dilates the support of the cutoff
function by a factor $d$. Since $\psi(y/d)$ is bounded by the sum of
$O(1)$ translates of $\psi(y)$, we may apply the definition of
$A_\la$. This only increases the constant by a bounded factor $C_d
$. (Moreover, observe that the new domain of
integration $\Delta_{b,h}$ is in the required form:
$\Delta_{b,h} = \bigcap_{i=1}^{dm} (\Delta - c_i)$ with $c_1 = 0$.)

\medskip

Let $J_\bbC (z, h) = J_\bbC (z, h_2,\cdots, h_d)$ denote, as before, the determinant of the holomorphic Jacobian matrix for the transformation $(z,h) = (z,h_2,\cdots, h_d) \mapsto \Gamma(z, h)$. Then Lemma \ref{jacest} implies that
\begin{align}\label{Jacobian-a-d}
J_\bbR (z, h) &= |J_\bbC (z, h)|^2\\ \notag
 &\ge c_{d,N} \, v(h)^2 \cdot {1\over d}\sum_{j=1}^d w_b (z+ h_j)^\frac{d^2+d}{2} \ge c_{d,N} \, v(h)^2 \prod_{j=1}^d w_b (z+ h_j)^\frac{d+1}{2}
\end{align}
for $z \in \Delta_{b,h} = \bigcap_{j=1}^d \bigcap_{i=1}^m(\Delta - b_i - h_j)$.

We also have
\begin{equation*}\label{}
|\tau(z,h)| = |\det (\Gamma'(z,h), \cdots, \Gamma^{(d)}(z,h))| = {1\over m}\, \Big| \sum_{i=1}^d \sum_{j=1}^m \phi
^{(d)} (z+ b_j+ h_i) \Big| .
\end{equation*}
Thus, as in the proof of Lemma \ref{jacest} we obtain
\begin{equation}\label{tauzh-d}
 |\tau(z,h)| \ge c_{d,N} \sum_{i=1}^d {1\over m}\, \sum_{j=1}^m |\phi
^{(d)} (z+ b_j+ h_i)| \ge c \max_{i=1,\cdots,d} w_b (z+ h_i)^\frac{d^2+d}{4}
\end{equation}
for $z \in \Delta_{b,h}$.
%
%

The estimates \eqref{tauzh-d} and \eqref{Jacobian-a-d} correspond to \eqref{Jacobian-a} (or \eqref{Jacbound}) and \eqref{tauzh} (or \eqref{torbound}), respectively, in the proof given in Section 6. (Note that here we need to keep track of the $b_i$'s unlike in the previous section. This is because only a weak form of a Jacobian bound, i.e. Lemma \ref{jacest}, is available in this context.)

The rest of the argument is the same as that in section 6. \qed


\medskip

\noindent{\sl Acknowledgements.}
The first-named author would like to thank
Dan Oberlin and Andreas Seeger for many helpful conversations about the subject matter of the paper. This paper is a by-product of a long-term collaboration with them. We thank also the anonymous referee for an earlier version of this paper for many useful suggestions which greatly helped improve our exposition.

\end{document}